\documentclass[12pt, a4paper]{article}
\usepackage[margin=1in]{geometry}

\usepackage{tikz,picins}
\usepackage{framed}
\usepackage{xparse}
\usepackage{amsthm,amsmath,amsfonts,amssymb}
\usetikzlibrary{positioning,calc}
\usetikzlibrary{arrows}
\usepackage{lmodern}
\usepackage[T1]{fontenc}
\usepackage{bigints}
\usepackage{mathtools}

\tikzstyle{myedgestyle} = [-latex]
\newtheorem{thm}{Theorem}[section]
\newtheorem{lem}[thm]{Lemma}

\newtheorem{cor}[thm]{Corollary}

\newtheorem{mydef}{Definition}[section]

\newtheorem{rem}{Remark}[section]
\usepackage[10pt]{moresize}
\usepackage{framed}
\usepackage{amsmath}
\theoremstyle{remark}
\newcommand\bb[1]{\mathbf{#1}}
\newcommand\ddfrac[2]{\frac{\displaystyle #1}{\displaystyle #2}}
\newcommand\bint[1]{\displaystyle\int #1}

\usepackage[font=footnotesize,labelfont=bf]{caption}
\usepackage{mathtools}

\def\dt{{\Delta t}}

\begin{document}

\author{  Sr\dj{}an Trifunovi\'c\thanks{School of Mathematical Sciences,
    Shanghai Jiao Tong University, Shanghai 200240, China,
    email: sergej1922@gmail.com, tarathis@sjtu.edu.cn} \and Ya-Guang Wang\thanks{School of Mathematical Sciences,
    MOE-LSC and SHL-MAC, Shanghai Jiao Tong University, Shanghai 200240,  China,
    email: ygwang@sjtu.edu.cn} }

\title{Existence of a weak solution to the fluid-structure interaction problem in 3D}

\maketitle
\abstract{{\footnotesize We study a nonlinear fluid-structure interaction problem in which the fluid is described by the three-dimensional incompressible Navier-Stokes equations, and the elastic structure is modeled by the nonlinear plate equation which includes a generalization of Kirchhoff, von K\'{a}rm\'{a}n and Berger plate models. The fluid and the structure are fully coupled via kinematic and dynamic boundary conditions. The existence of a weak solution is obtained by designing a hybrid approximation scheme that successfully deals with the nonlinearities of the system. We combine time-discretization and operator splitting to create two sub-problems, one piece-wise stationary for the fluid and one in the Galerkin basis for the plate. To guarantee the convergence of approximate solutions to a weak solution, a sufficient condition is given on the number of time discretization sub-intervals in every step in a form of dependence with number of the Galerkin basis functions and nonlinearity order of the plate equation.}}
\\ \\
{\footnotesize \textbf{Keywords and phrases:} {fluid-structure interaction, incompressible viscous fluid, nonlinear plate, three space variables, weak solution}
\\ \\
{\footnotesize \textbf{AMS Mathematical Subject classification (2010):} {35Q30, 35M13, 35D30, 74F10}
\section{Introduction}

In recent years there have been many works in the study of mathematical theory  of fluid-structure interaction (FSI) problems. These problems  arise from research fields like hydroelasticity, aeroelasticity, biomechanics, blood flow modeling and so on.
Chambolle et. al. in \cite{time} obtained the existence of a weak solution to the problem of viscous incompressible fluid and viscoelastic plate interaction in 3D, and Grandmont (\cite{grandmont3}) studied the limiting problem when viscoelasticity coefficient tends to zero. In [23 - 27] Muha and {\v C}ani{\'c} obtained the existence of weak solutions to several FSI problems by using the time discretization via operator splitting method. In \cite{BorSun} they studied the interaction in 2D case, and in \cite{multilayer,BorSunNonLinear} the 3D cylindrical case where the structure is described by linear and nonlinear Koiter shell equations, respectively. In \cite{multilayer2}, they observed a multi-layer structure of blood vessel in 2D,  and in \cite{BorSunNavierSlip}, they studied the 2D model with Navier-slip condition on the fluid-structure interface, where some additional difficulties due to the loss of the trace regularity of unknowns were successfully tackled.

In this paper, we study the interaction problem of viscous incompressible fluid and an elastic structure, which is modeled by a nonlinear plate law that precisely describes certain nonlinear phenomena in plate dynamics. This plate model is an extension of the one studied by Chueshov in \cite{igor,igor2} in the sense that certain higher order derivatives of displacement are allowed to be in the nonlinear term of the plate equation. In his work, Chueshov studied a plate model which is a generalization of Kirchhoff, von K\'{a}rm\'{a}n and Berger plate models, and got the well-posedness of the fluid-structure interaction problem where fluid is described by linearized compressible Navier-Stokes or linearized Euler equations, respectively. Unfortunately, the methods used in \cite{igor,igor2} strongly rely on the linearity of the fluid equations. We need to develop a different approach for our nonlinear problem, which was inspired from the approaches for nonlinear plates and fluid-structure interaction problems. For this reason we constructed a novel \textbf{hybrid approximation scheme} that deals with the nonlinearities both in fluid and plate equations by using the time discretization via the operator splitting method and the Galerkin approximation for the plate. We then prove that when the number of time discretization sub-intervals in every step is sufficiently large (compared to the number of Galerkin basis functions, corresponding eigenvalues and some other parameters), we obtain the convergence of a subsequence of approximate solutions to a weak solution of the original nonlinear problem.

\section{Preliminaries and main result}
In this section, we will first describe the model and derive the energy equality in the classical sense. After that, we introduce the domain transformation (the LE mapping) which is used in redefining the problem on a fixed domain. At the end of the section, we derive the equation in a weak form, give the definition of weak solutions and state the main result.
\subsection{Model description}
Here we deal with the incompressible, viscous fluid interacting
with nonlinear plate. The plate displacement is described by a scalar function $\eta:\Gamma\to \mathbb{R}$, where $\Gamma \subset \mathbb{R}^2$ is a connected bounded domain with Lipschitz boundary. The fluid fills the domain between side walls, the plate and the bottom, i.e.
\begin{eqnarray*}
\Omega^\eta(t) = \{ (X,z) : X \in \Gamma, -1< z < \eta(t,X)\},
\end{eqnarray*}
We will denote the graph of $\eta$ as $\Gamma^\eta (t) = \{ (X,z) : X \in \Gamma, z=\eta(t, X) \}$ and the
side wall of the domain as $W = \{(X,z): X \in \partial \Gamma, -1<z<0\}$, where the plate boundary is assumed to be fixed as $z=0$ for all $x \in \partial \Gamma$. The entire rigid part of the boundary  $\partial \Omega^\eta(t)$ will be denoted as $\Sigma =
(\Gamma \times \{-1\}) \cup W $.\\

The plate displacement $\eta(t,X)$ will be described by
the nonlinear plate law:
\begin{eqnarray}\label{thefirstplate}
\partial_t^2 \eta+\mathcal{F}(\eta)+\Delta^2 \eta =f,
\end{eqnarray}
where $f$ is the force density in vertical direction that comes
from the fluid and $\mathcal{F}$ is a nonlinear function corresponding the nonlinear elastic force in various plate models (see assumptions given in \eqref{assumption0}-\eqref{assumption2}). The plate is considered to be clamped
\begin{eqnarray}\label{clamped}
\eta(t, x)=0,~~ \partial_\nu \eta(t, x)=0, ~~\text{for all} ~~x \in
\partial\Gamma, t \in (0,T),
\end{eqnarray}
where $\nu$ is the normal vector on $\partial \Gamma$, and supplemented with initial data
\begin{eqnarray}\label{initialdata1}
 \eta(0,\cdot) = \eta_0,~~\partial_t \eta(0,\cdot) = v_0.
\end{eqnarray}

The \textbf{fluid motion} is described by the Navier-Stokes
equations
\begin{eqnarray}\label{NEq}
\begin{rcases}
\partial_t \mathbf{u} + \mathbf{u} \cdot \nabla \mathbf{u} =  \nabla \cdot
\sigma , \\
~~~~~~~~~~\nabla \cdot \mathbf{u} =  0 ,
\end{rcases}
\text{ in } \Omega^\eta(t),~ t\in (0,T).
\end{eqnarray}
where $\bb{u}$ is the velocity of the
fluid, $\mathbf{\sigma} = -pI + 2\mu \mathbf{D}(\mathbf{u})$, $\mu$ is the kinematic
viscosity coefficient and $\mathbf{D}(\mathbf{u})= \frac{1}{2}(\nabla \mathbf{u}+ (\nabla
\mathbf{u})^{\tau})$. The boundary condition on $\Sigma$ for the velocity $\bb{u}$ is no-slip
\begin{equation}
\mathbf{u}=0, ~~\text{on}~~ \Sigma,
\end{equation}
and we are supplemented with the initial data
\begin{eqnarray}\label{initialdata2}
\mathbf{u}(0, \cdot) = \mathbf{u}_0.
\end{eqnarray}

The \textbf{coupling} between the fluid and plate is defined
by two sets of boundary conditions on $\Gamma^\eta(t)$. In the
Lagrangian framework, with $ X \in \Gamma$ and $t\in (0,T)$, they
read as:
\begin{itemize}
\item The kinematic condition:
\begin{eqnarray}\label{kincond}
\partial_t \eta(t,X) \mathbf{e_3}= \mathbf{u}(t,X,\eta(t,X))
\end{eqnarray}
where  $\mathbf{e_3}=(0,0,1)$.
\item The dynamic condition:
\begin{eqnarray}\label{dyncond}
f(t,X) = -S^\eta (\mathbf{\sigma} \nu^\eta) \cdot \mathbf{e_3}.
\end{eqnarray}

Here, $\nu^\eta$ is the unit normal vector on the boundary $\Gamma^\eta$, and $S^\eta(t,X)$ is the Jacobian of the transformation from Eulerian to Lagrangian coordinates of the plate
\begin{eqnarray*}
S^\eta(t,X)=\sqrt{1+\partial_x \eta(t,X)^2+\partial_y
\eta(t,X)^2},
\end{eqnarray*}
with $X=(x,y)$. 
\end{itemize}

With $\eqref{dyncond}$, the plate equation $\eqref{thefirstplate}$ then becomes:
\begin{eqnarray}\label{plateeq}
\partial_{t}^2\eta+\mathcal{F}(\eta)+\Delta^2 \eta =-S^\eta (\mathbf{\sigma} \nu^\eta) \cdot \mathbf{e_3},
\end{eqnarray}
The initial data given in $\eqref{initialdata1}$ and $\eqref{initialdata2}$ are assumed to satisfy the following compatibility conditions:\

\begin{align}
\mathbf{u_0} \in L^2(\Omega^{\eta_0})^3, \nabla \cdot \mathbf{u_0} = 0, &\text{ in }
\Omega^{\eta_0},\nonumber \\
\mathbf{u_0} \cdot \nu =0, &\text{ on } \Sigma, \nonumber \\
\mathbf{u_0}(X,\eta_0(X)) = v_0(X) \cdot \mathbf{e_3}, &\text{ on } \Gamma,\label{initialconditions} \\
\partial_\nu\eta_0(X)=\eta_0(X) = v_0(X)=0, &\text{ on } \partial \Gamma, \nonumber \\
\eta_0(X)>-1, &\text{ on } \Gamma, \nonumber
\end{align}
where $\Omega^{\eta_0}=\Omega^\eta(0)$.

\subsection{The energy identity for smooth solutions}
In order to define the weak formulation of the problem $\eqref{thefirstplate}$-$\eqref{initialconditions}$, we need to choose the functional setting. Our bounds will come from the energy inequality of the approximated system, and it will dictate the regularity of solutions.

We can derive it for smooth solutions in the following way. We multiply $\eqref{plateeq}$ with $\partial_t \eta$ and integrate over $\Gamma$. Then we multiply the equation $\eqref{NEq}$ with $\bb{u}$, and integrate over $\Omega^\eta(t)$. We sum these two equalities, integrate it over $(0,t)$, and by partial integration we obtain:
\begin{eqnarray}\label{eneq}
E(t)+\mu \int_0^t ||\bb{D}(\bb{\tau})||_{L^2(\Omega^\eta(\tau))}^2 d\tau+\int_0^t(\mathcal{F}(\eta(\tau)), \partial_t \eta(\tau)) d\tau = E(0),
\end{eqnarray}
where
\begin{eqnarray*}
E(t) := \frac{1}{2}\Big(||\bb{u}(t)||_{L^2(\Omega^\eta(t))}^2+||\eta(t)||_{H^2(\Gamma)}^2+||\partial_t \eta(t)||_{L^2(\Gamma)}^2\Big)
\end{eqnarray*}
Even though the term in $\eqref{eneq}$ with $\mathcal{F}$ isn\rq{}t necessarily positive, in the  case we will study, we will be able to bound it from bellow in a suitable way (see (A3) on page 6). Therefore, we will use this equality as an inspiration for the choice of function spaces in the following discussion.

\subsection{LE mapping of the domain and functional setting} 

In order to define the problem on the fixed domain
\begin{eqnarray*}
\Omega =  \{ (X,z) : X \in \Gamma, -1< z < 0 \},
\end{eqnarray*}
we define a family of the following Lagrangian Eulerian (LE) transformations:
\begin{eqnarray*}
A_\eta(t) : &&\Omega \to \Omega^\eta(t),\\
&&(X,z) \mapsto (X,(z+1)\eta(t,X)+z).
\end{eqnarray*}
This mapping is a bijection and its Jacobian, defined by
\begin{eqnarray*}
J^\eta(t,X,z):=\text{det} \nabla A_\eta(t,X,z) = 1+\eta(t,X),
\end{eqnarray*}
is well-defined as long as $\eta(t,X)>-1$ for any $X \in \Gamma$. We also define the LE velocity
\begin{eqnarray}
&\mathbf{w}^\eta: = \ddfrac{d}{dt}A_\eta = (z+1)\partial_t \eta \mathbf{e_3},&\label{ALEeq3} 
\end{eqnarray}
In order to define the weak solution on the fixed domain $\Omega$, we
need to transform the problem $\eqref{thefirstplate}$-$\eqref{initialconditions}$ to be defined on $\Omega$ by using $A_\eta$. For an arbitrary vector function  $\bb{f}$  defined  on $\Omega^\eta (t)$, we define $\bb{f}^\eta$ on $\Omega$ as pullback of $\bb{f}$ by $A_\eta$, i.e
\begin{eqnarray*}
\bb{f}^\eta(t,X,z):=\bb{f}(t,  A_\eta(t,X,z)), ~~ \text{ for } (X,z) \in \Omega 
\end{eqnarray*}
the gradient as the push forward by $A_\eta$
\begin{eqnarray}
& \nabla^\eta \mathbf{f}^\eta:=(\nabla \mathbf{f})^\eta = \nabla \mathbf{f}^\eta (\nabla A_\eta)^{-1},\label{ALEeq2}
\end{eqnarray}
with
\begin{eqnarray}
&(\nabla A_\eta)^{-1} = [\bb{e}_1, \bb{e}_2, \overline{A}_\eta ], ~~~~
\overline{A}_\eta = \ddfrac{1}{\eta+1}[-(z+1)\partial_x \eta, -(z+1)\partial_y \eta, 1]^T,
\end{eqnarray}
the symmetrized gradient as
\begin{eqnarray*}
\mathbf{D}^\eta(\mathbf{f}^\eta):=\frac{1}{2}(\nabla^\eta \mathbf{f}^\eta+
(\nabla^\eta \mathbf{f}^\eta)^\tau ),\label{ALEeq5}
\end{eqnarray*}
and the LE derivative as a time derivative on the fixed domain $\Omega$
\begin{eqnarray}
&\partial_t \mathbf{f}_{|\Omega} := \partial_t \mathbf{f} + (\mathbf{w}^\eta \cdot \nabla)\mathbf{f}. & \label{ALEeq1} 
\end{eqnarray}
Now we write the Navier-Stokes equations in LE formulation by using the LE mapping,
\begin{eqnarray}\label{NStransformed}
\partial_t \mathbf{u}_{|\Omega} + (\bb{u} - \mathbf{w}^\eta)  \cdot\nabla \mathbf{u}&=& \nabla \cdot \sigma,~~ \text{ in } \Omega^\eta(t)
\end{eqnarray}
where $\partial_t \bb{u}_{|\Omega}$ and $\bb{w}^\eta$ are composed with the $(A_\eta(t))^{-1}$, and the transformed divergence free condition
\begin{eqnarray}\label{divfree}
\nabla^\eta \cdot \bb{u}^{\eta} &=& 0, ~ \text{ in } \Omega.
\end{eqnarray}

From the energy equality $\eqref{eneq}$, the possible regularity of $\eta$ is $H_0^2(\Gamma)$, and since $H_0^2(\Gamma)$ is embedded into the H\"{o}lder space $C^{0,\alpha}$ for $\alpha<1$, $\Omega^\eta(t)$ doesn\rq{}t necessarily have the Lipschitz boundary. For this lower regularity boundary we define the \lq\lq{}Lagrangian\rq\rq{} trace operator as
\begin{eqnarray*}
\gamma|_{\Gamma(t)} :&& C^1 \overline{(\Omega^\eta(t))}\to C( \Gamma), \\
 &&f \mapsto f(t,X, \eta(t,X)).
\end{eqnarray*}
In \cite{time, igor2, Boris}, it was proved that by continuity we can extend this trace operator $\gamma$ to be a linear operator from $H^1(\Omega^\eta(t))$ to $H^s(\Gamma)$, $0 \leq s<1/2$.  We first define the fluid velocity space
\begin{eqnarray*}
V_F := \{ \mathbf{u}= (u_1, u_2,u_3) \in C^1(\overline{\Omega^\eta(t)}): \nabla \cdot \mathbf{u}= 0, ~ \mathbf{u}= 0 ~ \text{ on } \Sigma \}
\end{eqnarray*}
and its following closure
\begin{eqnarray*}
\mathcal{V}_F: = \overline{V_F}^{H^1(\Omega^\eta(t))}.
\end{eqnarray*}
It is easy to have the following characterization (see  \cite{time, igor2}):
\begin{eqnarray*}
\mathcal{V}_F = \{ \mathbf{u} = (u_1, u_2,u_3) \in H^1(\overline{\Omega^\eta(t)}): \nabla \cdot \mathbf{u}= 0, ~ \mathbf{u}= 0 ~ \text{ on } \Sigma \}.
\end{eqnarray*}
The transformation of the domain, $A_\eta$, isn\rq{}t necessarily Lipschitz, so the transformed velocity $\bb{u}^\eta$ may not be in $H^1(\Omega)$. The transformed velocity space is then defined as
\begin{eqnarray*}
\mathcal{V}_F^\eta := \{\bb{u}^\eta: \bb{u} \in \mathcal{V}_F  \}
\end{eqnarray*}
Notice that any function $\bb{u}^\eta \in \mathcal{V}_F^\eta$ satisfies the transformed divergence free condition $\eqref{divfree}$ rather than the regular divergence free condition. When the Jacobian $J = \eta+1>0$, the inner product in $\mathcal{V}_F^\eta$ is defined as:
\begin{eqnarray*}
(\bb{f}^\eta,\bb{g}^\eta)_{\mathcal{V}_F^\eta}: = (\bb{f},\bb{g})_{H^1(\Omega^\eta (t))}.
\end{eqnarray*}
Now, we define the corresponding function space for fluid velocity involving time 
\begin{eqnarray*}
\mathcal{W}_F^\eta (0,T): =  L^\infty (0,T; L^2(\Omega)) \cap
L^2 (0,T; \mathcal{V}_F^\eta),
\end{eqnarray*}
while for the structure we choose classical space
\begin{eqnarray*}
\mathcal{W}_S(0,T) := W^{1, \infty}(0,T; L^2(\Gamma))\cap
L^\infty (0,T; H_0^2(\Gamma)). 
\end{eqnarray*}
We include the kinematic condition in the solution space
\begin{eqnarray*}
\mathcal{W}^\eta (0,T) :=  \{ (\bb{u}^\eta,\eta) \in
\mathcal{W}_F^\eta (0,T) \times \mathcal{W}_S(0,T):
\mathbf{u}^\eta|_\Gamma = \partial_t \eta \mathbf{e_3} \}
\end{eqnarray*}
and similarly for the corresponding test function space
\begin{eqnarray*}
Q^\eta (0,T) :=  \{ (\mathbf{q}^\eta, \psi) \in C_c^1 ([0,T);
\mathcal{V}_F^\eta \times H_0^2(\Gamma)) : \mathbf{q}^\eta|_\Gamma=
\psi(t,z) \mathbf{e_3} \}.
\end{eqnarray*}

Before the end of this subsection, let us impose the following 
assumptions on the nonlinear term $\mathcal{F}$ of the plate equation:

\begin{enumerate}
\item[(A1)] Mapping $\mathcal{F}$ is locally Lipschitz from $H_0^{2-\epsilon}(\Gamma)$ into $H^{-2}(\Gamma)$ for some $\epsilon >0$, i.e.
\begin{eqnarray}\label{assumption0}
||\mathcal{F}(\eta_1)-\mathcal{F}(\eta_2)||_{H^{-2}} \leq
C_R || \eta_1 -\eta_2 ||_{H^{2-\epsilon}(\Gamma)},
\end{eqnarray}
for a constant $C_R>0$, for any $||\eta_i||_{H^{2-\epsilon}(\Gamma)} \leq R$ ($i=1,2$). 

\item[(A2)] Mapping $\mathcal{F}$ is locally Lipschitz from $H_0^{2}(\Gamma)$ into $H^{-a}(\Gamma)$ for some $0\leq a<2$, i.e.
\begin{eqnarray}\label{assumption1}
||\mathcal{F}(\eta_1)-\mathcal{F}(\eta_2)||_{H^{-a}} \leq
C_R || \eta_1 -\eta_2 ||_{H^{2}(\Gamma)},
\end{eqnarray}
for a constant $C_R>0$, for any $||\eta_i||_{H^2(\Gamma)} \leq R$ ($i=1,2$). 

\item[(A3)]  $\mathcal{F}(\eta)$ has a potential
in $H_0^2(\Gamma)$, i.e. there exists a Fr\'{e}chet
differentiable functional $\Pi(\eta)$ on $H_0^2(\Gamma)$ such that
$\Pi'(\eta)=\mathcal{F}(\eta)$, and there are $0<\kappa <1/2$ and  $C^* \geq 0$, such that the following inequality holds,
\begin{eqnarray}\label{assumption2}
 \kappa || \Delta
\eta||_{L^2(\Gamma)}^2+\Pi(\eta)+ C^*\geq 0, ~~ \text{for all } \eta \in H_0^2(\Gamma).
\end{eqnarray}
Moreover, potential $\Pi(\eta)$ is bounded on bounded sets in $H_0^2(\Gamma)$. We will denote this bound by $C(\Pi, \eta)$, i.e. $\Pi(\tilde{\eta})\leq C(\Pi, \eta)$ for all $\tilde{\eta} \in H_0^2(\Gamma)$ such that $||\tilde{\eta}||_{H^2(\Gamma)} \leq ||\eta||_{H^2(\Gamma)}$.
\end{enumerate}

\begin{rem} 
{\normalfont (1) The assumption (A1) will be used to pass the convergence in the nonlinear term $\mathcal{F}(\eta)$ when the bound of approximate solutions is obtained in $H_0^2(\Gamma)$.  On the other hand the condition (A2) tells us the order of nonlinearity precisely - the function $\mathcal{F}(\eta)$ may depend on the derivatives of $\eta$ up to order $2+a$, and the parameter $a$ will also affect the minimal precision in time (the number of time sub-intervals) we will require for the approximate solution convergence (see inequality $\eqref{key}$). The coercivity condition $\eqref{assumption2}$ in the assumption (A3) is used to eliminate potential in the energy as it can be negative, while the bound of the initial potential $C(\Pi, \eta_0)$ defined in (A3) is used in bounding of the initial energy of the approximate problem in a finite Galerkin basis (see section 3.1.1).}

{\normalfont (2) It was proved that in several plate models such as von K\'{a}rm\'{a}n, Kirchhoff and Berger plates, the nonlinear elastic force $\mathcal{F}$ satisfies assumptions (A1)-(A3). In Appendix A more details can be found about the explicit forms of these plate models and the proofs that they satisfy these assumptions (an interested reader can also see \cite{igor, igor2,platesproofs} and the references therein).} 
\end{rem}

\subsection{The weak solution formulation and main result}

Before defining the weak solution for the problem $\eqref{thefirstplate}$-$\eqref{initialconditions}$, we first calculate some terms. For any given $(\bb{u},\eta) \in \mathcal{W}^\eta(0,T)$ and $(\bb{q},\psi) \in Q^\eta(0,T)$, multiplying the equation $\eqref{NStransformed}$ by $\bb{q}$ and integrating over $\Omega^\eta(t)$,  the convective term can be computed in the following way:
$$
\begin{array}{ll}
\int_{\Omega^\eta(t)}((\mathbf{u} - \mathbf{w}^\eta)\cdot
\nabla)\mathbf{u} \cdot \mathbf{q} =  & \frac{1}{2}\int_{\Omega^\eta(t)}( (\mathbf{u} - \mathbf{w}^\eta)\cdot \nabla)\mathbf{u} \cdot \mathbf{q}- \frac{1}{2}\int_{\Omega^\eta(t)} ((\mathbf{u}-\mathbf{w}^\eta ) \cdot \nabla) \bb{q} \cdot
\mathbf{u} \\[2mm]
& + \frac{1}{2}\int_{\Omega^\eta(t)}(\nabla \cdot \bb{w}^\eta )\mathbf{u} \cdot \mathbf{q} +  \frac{1}{2} \int_{\Gamma^\eta(t)}(\mathbf{u} - \mathbf{w}^\eta)\cdot \bb{\nu}^\eta (\bb{u} \cdot \bb{q}).
\end{array}
$$
in which the last term vanishes due to the kinematic coupling condition $\eqref{kincond}$. The diffusive part satisfies
\begin{eqnarray*}
-\int_{\Omega^\eta(t)} (\nabla \cdot \sigma) \cdot \bb{q} = 2\mu \int_{\Omega^\eta(t)} \bb{D}(\bb{u}):\bb{D}(\bb{q}) - \int_{\partial \Omega^\eta(t)} \sigma \bb{\nu}^\eta \cdot \bb{q},
\end{eqnarray*}
where the last term can be expressed as
\begin{eqnarray*}
\int_{\partial \Omega^\eta(t)} \sigma \bb{\nu}^\eta \cdot \bb{q} = \int_{\Gamma^\eta(t)}(\sigma \nu^\eta \cdot \bb{e}_3)\bb{q}\cdot \bb{e}_3 = \int_{\Gamma} S^\eta (\sigma \nu^\eta) \cdot \bb{e}_3 \psi.
\end{eqnarray*}
The only remaining term is the one including time derivative. Since we now want to express these integrals on the fixed domain $\Omega$, we calculate:
$$
\begin{array}{ll}
\int_0^T \int_{\Omega^\eta(t)} \partial_t \bb{u}^\eta_{|\Omega} \cdot \bb{q} &=  \int_0^T\int_{\Omega} J^\eta \partial_t \bb{u}^\eta \cdot q^\eta  \\
& = -\int_0^T\int_{\Omega} \partial_t J^\eta \bb{u}^\eta \cdot \bb{q}^\eta - \int_0^T\int_{\Omega} J^\eta \bb{u}^\eta \cdot \partial_t \bb{q}^\eta - \int_0^T\int_\Omega J_0 \bb{u}_0 \cdot \bb{q}^\eta(0, \cdot). 
\end{array}
$$
As in \cite[pp. 77]{gurtin}, by a simple calculation we have
\begin{eqnarray*}
 \partial_t J^\eta =\partial_t \eta =(1+\eta )\partial_z \Big( \frac{z-\eta }{1+\eta}+1 \Big)\partial_t \eta  = J^\eta \nabla^\eta \cdot \bb{w}^\eta.
\end{eqnarray*}

We finally multiply the plate equation $\eqref{plateeq}$ with $\psi$ and integrate over $(0, T)\times \Gamma$, and use partial integration in time. Summing this with the fluid equation $\eqref{NStransformed}$ multiplied with $\bb{q}$ and integrated over $(0, T)\times \Omega$ and using the calculation we just obtained, it leads to define:

\begin{mydef}\label{weaksolution}(\textbf{Weak solution}) Under the assumptions (A1)-(A3) of $\mathcal{F}$, we say that $ (\mathbf{u} , \eta ) \in
\mathcal{W}^\eta (0,T)$ is a weak solution of the problem $\eqref{thefirstplate}$-$\eqref{initialconditions}$ defined on the reference
domain $\Omega$, if the initial data $\bb{u}_0, \eta_0$ satisfy the compatibility conditions  given in $\eqref{initialconditions}$,  the following identity  
\begin{equation}
\begin{array}{l}
\frac{1}{2} \int_0^T \int_\Omega\Big( J^\eta \big(((\mathbf{u} - \mathbf{w}^\eta)\cdot
\nabla^\eta)\mathbf{u} \cdot \mathbf{q} -((\mathbf{u}-\mathbf{w}^\eta ) \cdot \nabla^\eta) \bb{q} \cdot
\mathbf{u} \big)- \partial_t J^\eta \mathbf{q} \cdot \mathbf{u}\Big)\\[2mm]
  +\int_0^T  \int_\Omega 2\mu J^\eta \mathbf{D}^\eta (\mathbf{u}):\mathbf{D}^\eta(\mathbf{q})
-\int_0^T \int_\Omega J^\eta \mathbf{u} \cdot \partial_t \mathbf{q} \\[2mm]
 -
\int_0^T \int_\Gamma \partial_t \eta \partial_t \psi dz dt +
\int_0^T (\mathcal{F} (\eta), \psi )+\int_0^T (\Delta \eta, \Delta \psi )\\[2mm]
=  \int_{\Omega} J_0 \mathbf{u_0} \cdot q(0) + \int_\Gamma v_0 \cdot \psi(0).
\end{array}
\end{equation}
holds for every $(\mathbf{q},\psi ) \in Q^\eta (0,T)$.

\end{mydef}

The main result of this paper is stated as follows:
\begin{thm}(\textbf{Main Theorem}). Assume that the nonlinear functional $\mathcal{F}$ satisfies the conditions given in (A1)-(A3). Then, for any given initial data $\eta_0 \in H_0^2(\Gamma), v_0 \in L^2(\Gamma)$ and $u_0 \in L^2(\Omega^{\eta_0})^3$ that satisfiy the compatibility conditions given in \eqref{initialconditions}, there exists a weak solution $(\bb{u},\eta) \in \mathcal{W}^\eta (0,T)$ of the problem $\eqref{thefirstplate}$-$\eqref{initialconditions}$ in the sense of Definition $\ref{weaksolution}$, satisfying the energy inequality:
\begin{eqnarray*}\label{en}
E(t)+\int_0^t ||\bb{D}(\bb{u}(\tau))||_{L^2(\Omega^\eta(\tau))}^2 d\tau \leq C_0 =   C (E(0) + C(\Pi, \eta_0) +C^*), \text{ for all } ~t\in(0,T)
\end{eqnarray*}
where 
\begin{eqnarray*}
E(t):=\frac{1}{2}\Big(||\bb{u}(t)||_{L^2(\Omega(t))}^2+||\eta(t)||_{H^2(\Gamma)}^2+||\partial_t \eta(t)||_{L^2(\Gamma)}^2\Big)
\end{eqnarray*}
The constants $C=1/(\frac{1}{2}-\kappa)$, $C^*$ and $C(\Pi,\eta_0)$ come from the boundedness of potential function $\Pi$ on bounded sets in $H_0^2(\Gamma)$ and the coercivity estimate $\eqref{assumption2}$, (see assumption $(A3)$). Moreover, this solution is defined on the time interval $(0,T)$ with $T=\infty$, or
$T<\infty$ which is the moment when the free boundary $\{ z = \eta(t,X)\}$ touches the bottom $\{ z= - 1\}$.
\end{thm}

The remainder of the paper is organized as follows. In section 3.1 we will formulate our approximate problems. The existence of approximate solutions will be given in section 3.2, and certain properties of these solutions will be discussed in section 3.3. In section 4,  we study the convergence of approximate solutions. The weak and strong convergences will be proved in sections 4.1 and 4.2, respectively. The proof of the main result will be given in section 5.


\section{Construction of approximate solutions}
In this section, by using the time discretization and Galerkin basis in $H_0^2(\Gamma)$, we will construct approximate solutions to the original problem $\eqref{thefirstplate}$-$\eqref{initialconditions}$, and obtain certain uniform estimates of approximate solutions. At the end of this section, the difference between $\partial_t \eta$ and the trace of the fluid velocity $v$ at $\Gamma^\eta$  (which are not equal from the approximate problems) is studied, and a sufficient condition  is introduced on the number of time sub-intervals $N=T/\Delta t$ in the time discretization in order to keep the difference smaller than $O((\Delta t)^\alpha)$ in $L^2(\Gamma)$ for some $\alpha >0$.

\subsection{Formulation of approximate problems}
For any fixed $T>0$ and $N\ge 1$, letting $\Delta t=\frac{T}{N} $, we split the time interval $[0,T]$ into $N$ equal 
sub-intervals and on each sub-interval we use the Lie operator splitting, and separate the problem into two parts -
the fluid and structure sub-problems. We rewrite the problem $\eqref{thefirstplate}$-$\eqref{initialconditions}$ as the following one:
\begin{eqnarray}
\frac{dX}{dt} &=& AX,~~ t \in (0,T), \\
X|_{t=0} &=& X^0,
\end{eqnarray}
where $X = (\mathbf{u},v,\eta)^T$, and $v =\partial_t \eta$ and
decompose $A = A_1+A_2$, where $A_1$ and $A_2$ are non-trivial and correspond to these two sub-problems. Since the sub-problems are not of the same nature, from here on we proceed to define them separately.

\subsubsection{The structure sub-problem (SSP) in the Galerkin basis}
First, we define the biharmonic eigenvalue problem as to find non-trivial $w \in H_0^2(\Gamma)$ and $\xi>0$ such that
\begin{eqnarray*}
\begin{cases}
 (\Delta w, \Delta f) = \xi (w, f),\\
w(X) = \partial_\nu w(X) = 0, \text{ on } \partial \Gamma,
\end{cases}
\end{eqnarray*}
for all $f \in H_0^2(\Gamma)$. This eigenvalue problem has a growing unbounded sequence of eigenvalues $0<\xi_1<\xi_2<...$, and corresponding smooth eigenfunctions $\{ w_i \}_{i \in \mathbb{N}}$ which form a basis of $H_0^2(\Gamma)$. The set $\{ w_i \}_{i \in \mathbb{N}}$ is called a Galerkin basis (see for example \cite[Theorem 7.22]{folland}).

Denote by $\mathcal{G}_{k}=$ span$(\{w_i\}_{ 1\leq
i \leq k})$ and the closed subspaces
\begin{eqnarray*}
&H_{0,k}^2(\Gamma): = (\mathcal{G}_k, ||\cdot||_{H^2(\Gamma)})\\
&L_k^2(\Gamma):= (\mathcal{G}_k, ||\cdot||_{L^2(\Gamma)}).
\end{eqnarray*}

\underline{The structure sub-problem (SSP):} \\

For any fixed $k \geq 1$, and $0 \leq n \leq N-1$, assume that $v^n \in L_k^2(\Gamma)$ is given already, we define $\eta^{n+1} \in C^1([n \Delta t, (n+1)\Delta t]; L_k^2)\cap C([n \Delta t, (n+1)\Delta t]; H_{0,k}^2(\Gamma))$ the
solution of the following problem inductively on $n$, 
\begin{eqnarray}\label{SSP}
\begin{cases} 
\Big(\ddfrac{\partial_t \eta^{n+1}(t) - v^n}{\Delta t},\psi
\Big)_\Gamma+(\Delta \eta^{n+1}(t), \Delta \psi)_\Gamma +
(\mathcal{F}(\eta^{n+1}(t)), \psi)_\Gamma = 0, \\ 
\eta^{n+1}(n \Delta t, X) =\eta^{n}(n \Delta t, X)
\end{cases}
\end{eqnarray}
for all $\psi \in H_{0,k}^2(\Gamma)$,  with $\eta_{\Delta t,k}^0(0,X) = \eta_{0,k}(X):= \sum_{i=1}^k (\eta_0, w_i) w_i$, 
where we have omitted $\Delta t$ and $k$ in the subscript of notation $\eta^{n+1}$.

The approximate solution $\eta_{\Delta t,k}(t,X)$ on whole
time interval $[0,T]$ is defined as the following way:
\begin{eqnarray*}
\eta_{\Delta t,k} (t) := \eta_{\Delta t,k}^{n+1}(t), ~~\text{for } t\in [n\Delta t, (n+1) \Delta t),~~ 0\leq n \leq
N-1.
\end{eqnarray*}

\subsubsection{The fluid sub-problem (FSP)}

Assume that we have $\eta_{\Delta t,k}^{n+1}$ already from the problem \eqref{SSP}, we define
following average quantity:
\begin{eqnarray*}
\widetilde{\partial}_t \eta_{\Delta t,k}^{n+1}: = \frac{1}{\Delta
t}\int_{n \Delta t}^{(n+1)\Delta t} \partial_t \eta_{\Delta t,k}^{n+1} dt= \frac{\eta_{\Delta t,k}^{n+1}((n+1) \Delta t)-\eta_{\Delta t,k}^n(n \Delta t)}{\Delta t}
\end{eqnarray*}
and the LE mapping $A_{\Delta t,k}^{n+1} = (X,
(z+1)\widetilde{\eta}_{\Delta t,k}^{n+1}+z))$, where $\widetilde{\eta}_{\Delta t,k}^{n+1} = \eta_{\Delta t,k}^{n+1}((n+1)\Delta t)$. The
discretizated Jacobian and LE velocity are given as:
\begin{eqnarray*}
 J_{\Delta t,k}^{n+1}:= \text{det} \nabla A_{\Delta t,k}^{n+1}=\widetilde{\eta}_{\Delta t,k}^{n+1}+1, ~~~
\bb{w}_{\Delta t,k}^{n+1}: =\widetilde{\partial}_t \eta_{\Delta t,k}^{n+1} \mathbf{e_3}
\end{eqnarray*}
To determine the approximate solution for the fluid part of the problem $\eqref{thefirstplate}$-$\eqref{initialconditions}$ in the weak form, we first discretize the time derivatives and obtain the following piece-wise stationary problem:
\begin{eqnarray}\label{FSPeq1}
\frac{\mathbf{u}^{n+1}-\mathbf{u}^n}{\Delta t }+((\mathbf{u}^n- \mathbf{w}^{n+1})\cdot
\nabla^{\eta^{n+1}})\mathbf{u}^{n+1} &=& \nabla^{\eta^{n+1}} \cdot
\mathbf{\sigma}^{\eta^{n+1}}(\mathbf{u}^{n+1},p^{n+1}),~ \text{ in } \Omega,~~~~ \label{FSP1}
\end{eqnarray}\

and
\begin{eqnarray}
\frac{\widetilde{\partial_t} v^{n+1}-\eta^{n+1}}{\Delta t}&=&
-S^{\eta^{n+1}} \mathbf{\sigma}^{\eta^{n+1}} (\mathbf{u}^{n+1}, p^{n+1})\cdot
\mathbf{\nu}^{\eta^{n+1}},\text{ on } \Gamma, \label{FSP2}\\
\mathbf{u}^{n+1}&=&v^{n+1} \mathbf{e_3}, \text{ on } \Gamma \label{FSPeq3},
\end{eqnarray}
with no-slip boundary condition on $\Sigma$, where we omitted $\Delta t,k$ in the subscript for simplicity.   

For a fixed basis of Galerkin functions for SSP, we
define the fluid problem function space as follows:
\begin{eqnarray*}
\mathcal{W}_k^{\tilde{\eta}^n} := \{ (\bb{u},v):  V_F^{\tilde{\eta}^n} \times
L_k^2(\Gamma), \bb{u}=v \mathbf{e_3} \text{ on } \Gamma\}.
\end{eqnarray*}
where
\begin{eqnarray*}
V_F^{\tilde{\eta}^n} := \{ \bb{u}:  \bb{u} \circ (A^{n})^{-1} \in H^1(\Omega^{\tilde{\eta}^{n}}),  \nabla^{\tilde{\eta}^n} \cdot \bb{u} = 0  \}
\end{eqnarray*}
For $\bb{u} \in V_F^{\tilde{\eta}^n}$, we will denote the corresponding mapped function $\hat{\bb{u}} := \bb{u} \circ (A^n)^{-1}$, and we introduce the inner product on $V_F^{\tilde{\eta}^n}$ as $(\bb{f},\bb{g})_{V_F^{\tilde{\eta}^n}}:= (\hat{\bb{f}},\hat{\bb{g}})_{H^1(\Omega^{\tilde{\eta}^{n}})}$.

For the initial data, we take 
$$v_{\Delta t,k}^0 =v_{0,k}:=\sum_{i=1}^k (v_0, w_i)w_i, \quad \mathbf{u}_{\Delta t,k}^0 =\mathbf{u}_{0,k}:= \mathbf{u}_0 \circ A_{\eta_0}-R_\Gamma(v_0-v_{\Delta t,k}^0)$$ 
where $R_\Gamma$ is a linear extension operator from $H^{s}(\Gamma)$, $s<1/2$, to $\mathcal{V}_F^{\eta_0}$ (see section 5.1 and extension of function $\bb{r}_1$). Notice that we had to modify $\mathbf{u_{0,k}}$ since the change of the initial displacement affects the compatibility conditions, and the above choice makes the compatibility conditions being preserved for the approximated problem.

Omitting ${\Delta t,k}$ in the subscript and simplifying the notation $\nabla^n := \nabla^{\eta^n}$, corresponding to $\eqref{FSPeq1}$-$\eqref{FSPeq3}$, we now define the following weak form of \underline{the fluid sub-problem (FSP)}:

Assume that $\widetilde{\eta}^{n+1} \in H_{0,k}^2(\Gamma)$ and $\widetilde{\partial}_t \eta^{n+1} \in L_k^2(\Gamma)$ are given already, the problem is to find $(\mathbf{u}^{n+1},v^{n+1}) \in\mathcal{W}_k^{\widetilde{\eta}^{n}}$ such that
\begin{align}
&\ddfrac{1}{2} \bint_\Omega J^n \big( ((\mathbf{u}^{n} - \mathbf{w}^{n+1})\cdot
\nabla^{n+1})\mathbf{u}^{n+1} \cdot \mathbf{q} - ((\mathbf{u}^n-\mathbf{w}^{n+1} ) \cdot
\nabla^{n+1}) \mathbf{q} \cdot \mathbf{u}^{n+1} \big) &\nonumber \\
&+\bint_{\Omega} J^n\ddfrac{\mathbf{u}^{n+1}-\mathbf{u}^n}{\Delta t} \cdot \mathbf{q} +
\frac{1}{2}\bint_\Omega \ddfrac{J^{n+1}-J^n}{\Delta t} \mathbf{u}^{n+1}
\cdot \mathbf{q}&  \nonumber  \\
&+2 \mu \int_\Omega J^{n} \mathbf{D}^{n+1} (\mathbf{u}^{n+1}):\mathbf{D}^{n+1}(\mathbf{q}) + \bint_\Gamma
\ddfrac{v^{n+1}-\widetilde{\partial}_t \eta^{n+1}}{\Delta t} \psi
=0, &\label{FSP}
\end{align}
for all $(\bb{q},\psi)\in \mathcal{W}_k^{\widetilde{\eta}^{n}}$. Notice that the solution and the test functions are defined in the space that correspond to the "previous" domain $\Omega^{\eta_{\Delta t,k}^n}$ (instead of $\Omega^{\eta_{\Delta t,k}^{n+1}}$).

We define piece-wise stationary solution on the whole time interval $[0,T]$:
\begin{eqnarray*}
\mathbf{u}_{\Delta t, k}(t):= \mathbf{u}_{\Delta t,k}^n,~~ v_{\Delta t, k}(t):= v_{\Delta t,k}^n, ~~\text{ for }
t\in [(n-1)\Delta t, n \Delta t),~~1 \leq n\leq N,
\end{eqnarray*}

We will sometimes use the notation $(\bb{u}^{n+1},v^{n+1})=$ FSP$(\widetilde{\partial}_t \eta^{n+1}, \widetilde{\eta}^{n+1}, \bb{u}^{n})$ to denote the solution of $\eqref{FSP}$ and for $N = T/\Delta t$ and $k$, we will write the solution
\begin{eqnarray*}
X_{N,k} = (\bb{u}_{\Delta t,k}, \eta_{\Delta t,k}, v_{\Delta t,k})^T
\end{eqnarray*}
determined by $\eqref{SSP}$ and $\eqref{FSP}$, inductively. \\


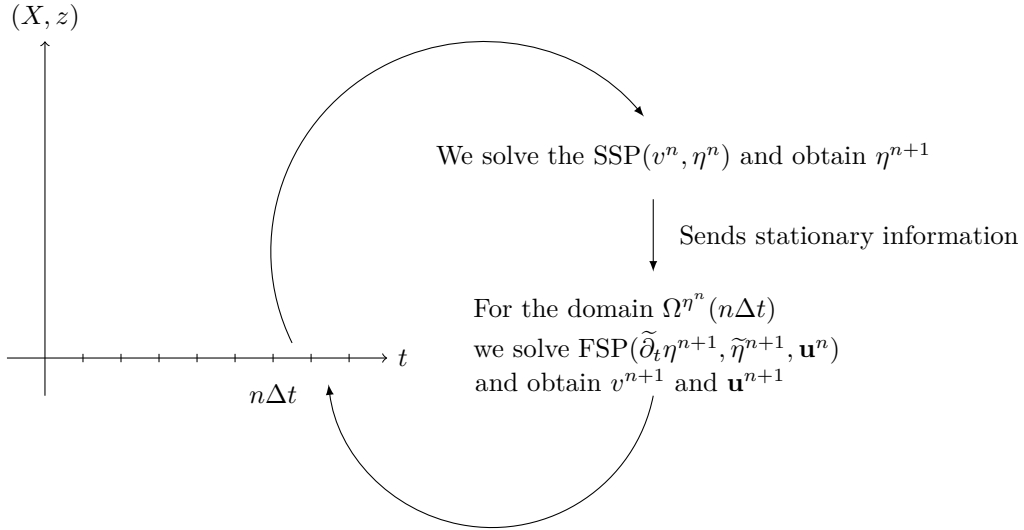
\begin{figure}[!h]
\begin{center}
\begin{tikzpicture}[domain=0:3]

      \draw[->] (-0.5,0) -- (4.5,0) node[right] {{\footnotesize $t$}};
      \draw[->] (0,-0.5) -- (0,4.2) node[above] {{\footnotesize$ (X,z)$}};

\node (v1) at (0.5,0.2) {};
\node (v2) at (0.5,-0.2) {};
\draw  (v1) edge (v2);

\node (v4) at (1,0.2) {};
\node (v3) at (1,-0.2) {};
\node (v5) at (1.5,0.2) {};
\node (v6) at (1.5,-0.2) {};
\node (v8) at (2,0.2) {};
\node (v7) at (2,-0.2) {};
\node (v10) at (2.5,0.2) {};
\node (v9) at (2.5,-0.2) {};
\node (v12) at (3,0.2) {};
\node (v11) at (3,-0.2) {};
\node (v14) at (3.5,0.2) {};
\node  (v13) at (3.5,-0.2) {};
\node (v16) at (4,0.2) {};
\node (v15) at (4,-0.2) {};
\draw  (v3) edge (v4);
\draw  (v5) edge (v6);
\draw  (v7) edge (v8);
\draw  (v9) edge (v10);
\draw  (v11) edge (v12);
\draw  (v13) edge (v14);
\draw  (v15) edge (v16);

\node [anchor=north] (v20) at (3,-0.2) {{\footnotesize $n \Delta t$}};

\node [anchor=north] at (8.4,3) {{\footnotesize We solve the SSP($v^{n},\eta^{n})$ 
and obtain $\eta^{n+1}$}};

\draw[-latex] (3.25,0.2) arc (-154.0988:-320:2.7792);

\node (v17) at (8,2.25) {};
\node [anchor=west] at (8.2,1.6) {\footnotesize Sends stationary information};
\node [anchor=north] at (7.63,1) {\footnotesize For the domain $\Omega^{\eta^{n}}(n\Delta t)$};

 \node [anchor=north] (v19) at (8.05,0.5) {{\footnotesize we solve
FSP($\widetilde{\partial}_t \eta^{n+1}, \widetilde{\eta}^{n+1}, \bb{u}^{n})$ }};
 \node [anchor=north] (v20) at (7.7,0) {{\footnotesize and obtain $v^{n+1}$ and $\bb{u}^{n+1}$}};
\node (v18) at (8,1) {};
\draw[-latex]  (v17) edge (v18);
\draw[-latex] (8,-0.5) arc (-11.0789:-173:2.1636);
\end{tikzpicture}
\caption{The diagram of the solving procedure.}
\end{center}
\end{figure}

\begin{rem} 
{\normalfont (1) We defined the structure sub-problem by using the Galerkin approximation and this way we created a hybrid approximation scheme, with fluid being piece-wise stationary and plate displacement being continuous in time (and plate displacement velocity being piece-wise continuous in time). Since these two sub-problems are communicating, we had to modify the data that structure sub-problem sends to fluid sub-problem from continuous in time to stationary, because the fluid sub-problem cannot be solved with continuous in time information. This creates a difference between functions and their averages which needs to be taken into consideration later. }

{\normalfont (2) In the original problem $\eqref{thefirstplate}$-$\eqref{initialconditions}$, the plate displacement velocity $\partial_t \eta \bb{e}_3$ and the trace of the fluid $\bb{u}(t,X,\eta(t,X))$ are equal due to the kinematic boundary condition $\eqref{kincond}$. However, in general, this is not true for approximate solutions determined from the approximate problems. In previous work [23 - 27], their difference goes to zero by an estimate derived from the discretized energy inequalities of sub-problems (see for example \cite[Proposition 4]{BorSunNavierSlip}), while in the present work the situation is more complicated. We will study this difference in detail  in section 3.3 by choosing a special discretization step depending on the number of Galerkin basis functions and some other parameters.}
\end{rem}

\subsection{Solutions of (SSP), (FSP) and discrete energy estimates}
We will first define appropriate version of energy for each
time interval and basis of functions for $n\Delta t\leq t\leq (n+1)\Delta t$:
$$
\begin{array}{ll}
S_{\Delta t, k}^n(t) =& \frac{1}{2\Delta t} \int_{n\Delta t}^{t} ||
\partial_t \eta_{\Delta t,
k}^n(\tau)||_{L^2(\Gamma)}^2 d\tau+\frac{1}{2}||\Delta \eta_{\Delta t,k}^n(t)||_{L^2(\Gamma)}^2+\Pi(\eta_{\Delta t, k}^n(t)) \\[2mm]
& +
\frac{t-n\Delta t}{2 \Delta t}\int_\Omega J^{n-1}|\mathbf{u}_{\Delta t, k}^{n-1}|^2 \\[2mm]
F_{\Delta t, k}^n(t) =& \frac{t-n\Delta t}{2\Delta t} || v_{\Delta t,
k}^n||_{L^2(\Gamma)}^2+\frac{1}{2}||\Delta \eta_{\Delta t,
k}^{n+1}(t)||_{L^2(\Gamma)}^2+\Pi(\eta_{\Delta t, k}^{n+1}(t)) \\[2mm]
&+
 \frac{t-n\Delta t}{2\Delta t} \int_\Omega J^n|\mathbf{u}_{\Delta t, k}^n|^2 \\[2mm]
D_{\Delta t,k}^n = &\Delta t \mu \int_{\Omega} J^{n-1} \mathbf{D}(\mathbf{u}_{\Delta
t,k}^n):\mathbf{D}(\mathbf{u}_{\Delta t, k}^n)
\end{array}
$$
We will omit ${\Delta t,k}$ of the subscript throughout most of this section. 

\begin{lem}\label{SSPestimate}
For a given function $v^n \in L^\infty(0,T ; L^2(\Gamma))$, the solution
$\eta^{n+1}(t) = \sum_{i=1}^k \alpha_{i,n+1}(t)w_i$ of (SSP)
satisfies the following a priori estimates:
\begin{equation}\label{lem11}
\ddfrac{1}{2\Delta t}\bint_{n \Delta t}^{t}||\partial_t
\eta^{n+1}-v^n||_{L^2(\Gamma)}^2+S^{n+1}(t) = F^n(n \Delta t),
\end{equation}
and
\begin{equation}\label{lem12}
\begin{array}{ll}
\ddfrac{t-n\Delta t}{2 \Delta t}\bint J^{n}|\mathbf{u}_{\Delta t, k}^{n}|^2 & + \frac{1}{2\Delta t}\bint_{n \Delta t}^{t} \big(||\partial_t
\eta^{n+1}-v^n||_{L^2(\Gamma)}^2+||\partial_t
\eta^{n+1}||_{L^2(\Gamma)}^2 \big) 
\\[2mm]
& + c||\Delta
\eta^{n+1}(t)||_{L^2(\Gamma)}^2  \leq C^*+ F^n(n \Delta t),
\end{array}
\end{equation}
with $n\Delta t \leq t \leq (n+1)\Delta t$, where $c = 1/2-\kappa$.

\end{lem}
\begin{proof}

In $\eqref{SSP}$, by taking $\psi=\alpha_{i,n+1}'(t)w_i$ and summing over
$i=1,...,k$, it follows
\begin{eqnarray*}
\frac{1}{2 \Delta t}\big(||\partial_t
\eta^{n+1}(t)-v^n||_{L^2(\Gamma)}^2+||\partial_t
\eta^{n+1}(t)||_{L^2(\Gamma)}^2\big)\\
+\frac{1}{2}\frac{d}{dt}||\Delta 
\eta^{n+1}(t)||_{L^2(\Gamma)}^2+ \frac{d}{dt} \Pi(\eta^{n+1}(t))= 
 \frac{1}{2 \Delta t}||v^n||_{L^2(\Gamma)}^2,
\end{eqnarray*}
by using $\frac{d}{dt} \Pi(\eta)=(\mathcal{F}(\eta),\partial_t
\eta)$, and $2(a-b)a=  ((a-b)^2+a^2-b^2)$. Integrating this equality on $[n\Delta t,t]$ and adding $\int J^n |\mathbf{u}^n|^2$ on both sides, we obtain
$\eqref{lem11}$, and using the coercivity property of potential $\eqref{assumption2}$, the estimate $\eqref{lem12}$ follows immediately. 

\end{proof}

\begin{lem}\label{sspsolution}
For a given function $v^n \in L^\infty(0,T ; L^2(\Gamma))$, the problem (SSP) has a unique solution 
$$\eta^{n+1}(t)=\sum_{i=1}^k \alpha_{i,n+1}(t)w_i$$ 
in $C^1([n\Delta t, (n+1)\Delta t];L_k^2(\Gamma))\cap C([n\Delta t, (n+1)\Delta t];
H_{0,k}^2(\Gamma))$.
\end{lem}

\begin{proof}
In the equation $\eqref{SSP}$, by choosing $\psi = w_i$ with $i=1,...,k$, we obtain  that $\alpha_{n+1}(t)=(\alpha_{1,n+1},...,\alpha_{k,n+1})^T$ satisfies the following problem,
\begin{eqnarray}\label{notsocloseeq}
\begin{cases}\alpha_{n+1}'(t)+ \Delta t \Xi
\alpha_{n+1}(t)+\Delta t \overline{\mathcal{F}}(\alpha_{n+1}(t)) = V^n, \\
\alpha_{n+1}(n\Delta t) = \big((\eta^{n}(n\Delta t),w_1),...,(\eta^{n}(n\Delta t),w_k)\big)^T \end{cases}
\end{eqnarray}
where $\Xi=\text{diag}(\xi_1,...,\xi_k)$, 
\begin{eqnarray*}
\overline{\mathcal{F}}(\alpha_{n+1}(t))= [(\mathcal{F}\big(\sum_{i=1}^k
\alpha_{i,n+1}(t)w_i\big),w_1), ..., (\mathcal{F}\big(\sum_{i=1}^k
\alpha_{i,n+1}(t)w_i\big),w_k)]^T,
\end{eqnarray*}
and 
\begin{eqnarray*}
V^n =\big( (v^n,w_1),....,(v^n, w_k)  \big)^T.
\end{eqnarray*}
To solve the above problem, let us verify that  $\overline{\mathcal{F}}(f(t))$ is a Lipschitz function with the Lipschitz constant being uniform in $t$. For two functions $f=\sum_{i=1}^n f_i(t) w_i$ and
$g=\sum_{i=1}^n g_i(t) w_i$ such that $||f||_{H^2(\Gamma)},||g||_{H^2(\Gamma)} \leq R$ we have:
\begin{eqnarray*}
|\overline{\mathcal{F}}(f(t))-\overline{\mathcal{F}}(g(t))| &\leq& C_{R}||\mathcal{F}(f(t))-\mathcal{F}(g(t))||_{{H^{-a}(\Gamma)}}\sum_{i=1}^k||w_i||_{H^a}
\\
&\leq&   C_{R}\big|\big|\sum_{i=1}^k
(f_i(t)-g_i(t))w_i\big|\big|_{H^{2}}\sum_{i=1}^k||w_i||_{H^a} \\
&\leq & C_{R} \xi_k k^2 |f(t)-g(t)|_\infty.
\end{eqnarray*}
with $|f(t)|_\infty= \max_{1\leq i\leq k} |f_i(t)|$, where we have used $||w_i||_{H^2}=\sqrt{\xi_k}$ from the eigenvalue equality $\xi_i ||w_i||_{L^2}^2 = ||w_i||_{H^2}^2$. Since the solution of the problem $\eqref{notsocloseeq}$ satisfies a priori estimate $\eqref{lem12}$, we can choose $R=(C^*+ F^n(n \Delta t))/c$ and obtain a uniform Lipschitz constant. Now, from the existence theory for ordinary differential equations, we obtain the unique solution $\alpha_{n+1}\in C^1[n\Delta t,n\Delta t+t_0]$ for certain $t_0 > 0$, and from the bound $\eqref{lem12}$ it follows that $t_0 =\Delta t$. This finishes the proof.
\end{proof}

\begin{lem}\label{fspsolution}
Let $J^n \geq c >0$. For a given function $\eta^{n+1}\in
C^1([n\Delta t, (n+1)\Delta t]; L_k^2(\Gamma))\cap \\ C([n \Delta t,(n+1)\Delta t]; H_{0,k}^2(\Gamma))$ there exists a
unique solution $(\mathbf{u}^{n+1},v^{n+1}) \in \mathcal{W}_k^{\tilde{\eta}^{n}}$ of
(FSP),  and it satisfies following inequality:
\begin{equation}\label{lem14}
F^{n+1}\big((n+1)\Delta t\big)+\frac{1}{2}\int_\Omega J^n |\mathbf{u}^{n+1}-\mathbf{u}^n|^2+\frac{1}{2}
||v^{n+1}-\widetilde{\partial}_t \eta^{n+1}||_{L^2(\Gamma)}^2 + D^{n+1} \leq S^{n+1}\big((n+1)\Delta t\big).
\end{equation}

\end{lem}
\begin{proof}
The existence of $(\mathbf{u}^{n+1},v^{n+1})$ to (FSP) can be proved in the same way as in \cite{BorSun} by using the Lax-Milgram Lemma. First we want to bound the dissipation term from bellow by $V_F^{\tilde{\eta}^{n}}$ norm. For this reason we map the velocity $\bb{u}^{n+1}$ back to the moving domain and by the Korn inequality (see \cite[Chapter 3,Theorem 3.1]{duvaitlions})
\begin{eqnarray*}
||\bb{u}^{n+1}||_{V_F^{\tilde{\eta}^{n}}}^2 \leq \frac{1}{\Delta t}  C_{n}  D^{n+1},
\end{eqnarray*}
where $C_{n}$ is Korn\rq{}s constant that depends on the domain $\Omega^{\tilde{\eta}^{n}}$.
The continuity property of two functionals (in the sense of the Lax-Milgram Lemma) and the lower bounds of the remaining terms for the coercivity property are proved straightforward, so the unique solution follows.

To derive the inequality $\eqref{lem14}$,  letting $\mathbf{q}= \mathbf{u}^{n+1}$ and $\psi=v^{n+1}$ in $\eqref{FSP}$, it follows
\begin{eqnarray*}
\int_{\Omega} J^n\frac{\mathbf{u}^{n+1}-\mathbf{u}^n}{\Delta t} \cdot \mathbf{u}^{n+1} +
\frac{1}{2}\int_\Omega \frac{J^{n+1}-J^n}{\Delta t} \mathbf{u}^{n+1}
\cdot \mathbf{u}^{n+1}  \\
=\frac{1}{2\Delta t} \int_\Omega \big( J^{n+1} |\mathbf{u}^{n+1}|^2+ J^{n} |\mathbf{u}^{n+1}-\mathbf{u}^{n}|^2-  J^{n} |\mathbf{u}^{n}|^2 \big).
\end{eqnarray*}

On the other hand, obviously we have 
\begin{eqnarray*}
\int_\Gamma
\frac{v^{n+1}-\widetilde{\partial}_t \eta^{n+1}}{\Delta t} v^{n+1}
=\frac{1}{2 \Delta t} \big(  ||v^{n+1}-\widetilde{\partial}_t \eta^{n+1}||_{L^2(\Gamma)}^2+||v^{n+1}||_{L^2(\Gamma)}^2-||\widetilde{\partial}_t \eta^{n+1}||_{L^2(\Gamma)}^2 \big),
\end{eqnarray*}
and
\begin{eqnarray*}
||\widetilde{\partial}_t \eta^{n+1}||_{L^2(\Gamma)}^2 \leq \frac{1}{\Delta t}
\int_0^{\Delta t}||\partial_t \eta^{n+1}||_{L^2(\Gamma)}^2,
\end{eqnarray*}
from the H\"{o}lder inequality. Thus,  we conclude the inequality $\eqref{lem14}$ immediately.
\end{proof}

Now we are ready to obtain the uniform bounds of the approximate solutions as follows.

\begin{lem}\label{111}
For a given $\Delta t>0$, $T = N \Delta t$, and $k \in
\mathbb{N}$, let $\eta_{\Delta t,k}(t), \mathbf{u}_{\Delta t,k}(t),
v_{\Delta t,k}(t)$ be the solutions of (SSP) and (FSP) given in Lemmas $\ref{sspsolution}$ and $\ref{fspsolution}$ respectively on the time interval $[0,T]$. Then, one has:
\begin{enumerate}
\item[(1)] for all $0 \leq n \leq N$ and $n\Delta t \leq t \leq (n+1)\Delta t$, 
$$F_{\Delta t, k}^n(t),S_{\Delta t,k}^n(t) \leq C_0 =E(0) + C(\Pi, \eta_0), \quad F_{\Delta t,k}^n(n\Delta t)+\sum_{i=1}^n D_{\Delta t,k}^i\leq C_0$$  
where $E(0)=\frac{1}{2}\Big(||\bb{u}_0||_{L^2(\Omega(t))}^2+||\eta_0||_{H^2(\Gamma)}^2+||v_0||_{L^2(\Gamma)}^2\Big)$ is the initial energy, and $C(\Pi, \eta_0)$ is defined in (A3);
\item[(2)] $\eta_{\Delta t,k}$ is bounded in $L^\infty(0,T; H_0^2(\Gamma))$ uniformly with respect to $\Delta t, k$;
\item[(3)] $\partial_t \eta_{\Delta t,k }$ is bounded in $L^2(0,T;
L^2(\Gamma))$ uniformly with respect to $\Delta t, k$;
\item[(4)] $v_{\Delta t,k}$ is bounded in $L^\infty(0,T; L^2(\Gamma))$ uniformly with respect to $\Delta t, k$;
\item[(5)] $\mathbf{u}_{\Delta t,k}$ is bounded in $L^\infty(0,T ; L^2(\Omega))$ uniformly with respect to $\Delta t, k$;
\item[(6)] the following estimate
$$
\begin{array}{l}
\displaystyle\sum_{n=0}^{N-1} \Big(\ddfrac{1}{\Delta t} \bint_{n\Delta t}^{(n+1)\Delta t}||\partial_t
\eta_{\Delta t, k}^{n+1}-v_{\Delta t, k}^n||_{L^2(\Gamma)}^2+||v_{\Delta
t, k}^{n+1}-\widetilde{\partial}_t \eta_{\Delta t,
k}^{n+1}||_{L^2(\Gamma)}^2+\int_\Omega J_{\Delta t, k}^n |\mathbf{u}_{\Delta t,
k}^{n+1}-\mathbf{u}_{\Delta t, k}^n|^2 \Big) \\[2mm]
+\sum_{n=1}^N D_{\Delta t, k}^n   \leq C
\end{array}
$$
holds.

\end{enumerate}
\end{lem}

\begin{proof}
The estimates given in (1) follow from $\eqref{lem11}$ and $\eqref{lem14}$. The boundedness given in (2) and (3) comes from $\eqref{lem12}$ and first inequality given in (1). By taking out $\Pi(\eta^{n+1}((n+1)\Delta t))$ in $\eqref{lem14}$ from both sides, we obtain
\begin{eqnarray*}
\frac{1}{2} \Big( \int J^n |\mathbf{u}_{\Delta t,k}^{n+1}-\mathbf{u}_{\Delta t,k}^n|^2+
\int_\Omega J^{n+1}|\mathbf{u}_{\Delta t, k}^{n+1}|^2+||v_{\Delta t,k}^{n+1}-\widetilde{\partial}_t \eta_{\Delta t,k}^{n+1}||_{L^2(\Gamma)}^2+||v_{\Delta t,k}^{n+1}||_{L^2(\Gamma)}^2 \Big) +D_{\Delta t,k}^{n+1}\\
\leq S^{n+1}\big((n+1)\Delta t\big)-\Pi\big(\eta_{\Delta t,k}^{n+1}((n+1)\Delta t)\big) \leq \frac{1}{c}(C^*+F^n),
\end{eqnarray*}
where last inequality comes from $\eqref{lem12}$, with $C^*$ being given in the coercivity estimate from (A3). This yields the boundedness given in (4) and (5). The estimate given in (6) is obtained by summing $\eqref{lem11}$ and $\eqref{lem14}$ over all $1\leq n \leq N$ and by telescoping. 
\end{proof}

\subsection{Estimate of $\partial_t \eta_{\Delta t,k}-v_{\Delta t,k}$}
Since we are working with two parameters $k$ and $N = T/\Delta t$ in constructing the approximate solutions, we want to pass the convergence of the approximate solutions at the same time in both parameters. In order to ensure that $\partial_t \eta_{\Delta t,k}$ and $v_{\Delta t,k}$  converge to the same function in $L^2(0,T; L^2(\Gamma))$, we want $N = N(k)$ to be sufficiently large in every step. As a preparation, we first derive some estimates on $\partial_t \eta_{\Delta t,k}-v_{\Delta t,k}$.

For any function $f \in H_{0,k}^2(\Gamma)$, by using the orthogonality of $\Delta w_1, ... ,\Delta w_k$ in $L^2(\Gamma)$, we have:
\begin{equation}\label{so}
||f||_{H^2}^2 \leq C_\Gamma||\Delta f||_{L^2}^2= C_\Gamma\big(\sum_{i=1}^k (f, w_i)\Delta w_i\big)^2 =  C_\Gamma\sum_{i=1}^k(f, w_i)^2 (\Delta w_i, \Delta w_i) = C_\Gamma \sum_{i=1}^k \xi_i(f, w_i)^2.
\end{equation}
Now taking $\psi=w_i$ in $\eqref{SSP}$, we obtain the inequality 
\begin{align}
\frac{1}{(\Delta t)^2} ( \partial_t \eta^{n+1}(t)-v^n,w_i)^2
&\leq  2 (\Delta \eta^{n+1}(t), \Delta w_i)^2+ 2 (
\mathcal{F}(\eta^{n+1}(t))-\mathcal{F}(0)+\mathcal{F}(0), w_i)^2 \nonumber \\ 
&\leq
2 \xi_i^2 (\eta^{n+1}(t),  w_i)^2+2C_{R}^2\big(||\mathcal{F}(0)||_{H^{-a}}+||\eta^{n+1}(t)||_{H^2}\big)^2 ||w_i||_{H^a}^2. \label{so4}
\end{align}
by using the Lipschitz continuity of $\mathcal{F}$ given in the assumption (A2), where $C_R$ is 
the Lipschitz constant given in (A2), which is uniform due to the boundedness of $\eta_{\Delta t,k}$ given in Lemma \ref{111}(2). 

Summing $\eqref{so4}$ over $i=1,...,k$, it follows
\begin{align}
\frac{1}{(\Delta t)^2} || \partial_t \eta^{n+1}(t)-v^n||_{L^2(\Gamma)}^2
&\leq
2\sum_{i=1}^k \xi_i^2(\eta^{n+1}(t),w_i)^2  +2C_{R}^2\big(||\mathcal{F}(0)||_{H^{-a}}+||\eta^{n+1}(t)||_{H^2}\big)^2\sum_{i=1}^k||w_i||_{H^a}^2& \nonumber \\
&\leq 2 \xi_k ||\Delta \eta^{n+1}||_{L^2}^2 + 2C_{R}^2\big(||\mathcal{F}(0)||_{H^{-a}}+||\eta^{n+1}(t)||_{H^2}\big)^2\sum_{i=1}^k||w_i||_{H^a}^2,& \label{so3}
\end{align}
by using $\eqref{so}$. 

From the interpolation inequality for the Sobolev spaces, we have
\begin{eqnarray}\label{so8}
||w_n||_{H^a}^2 \leq ||w_n||_{L^2(\Gamma)}^{{2-a}}||w_n||_{H^2}^{{a}} \leq C_\Gamma\xi_i^{a/2},
\end{eqnarray}
and from the uniform bound for $\eta_{\Delta t,k}$ given in Lemma \ref{111}(2), we obtain
\begin{align}
\frac{1}{(\Delta t)^2} || \partial_t \eta^{n+1}(t)-v^n||_{L^2}^2 \leq  
2\xi_k \frac{C_\Gamma}{c}(C^*+C_0)+ 4C_R^2 C_\Gamma \big(||\mathcal{F}(0)||_{H^{-a}}^2 + \frac{C_\Gamma}{c}(C^*+C_0)\big) \sum_{i=1}^k \xi_i^{a/2}  \label{so5}
\end{align}
where $c = (\frac{1}{2}-\kappa)$, with $\kappa$ and $C^*$ being the constants given in the coercivity condition of \eqref{assumption2}. Denoting by
\begin{eqnarray}\label{CB}
C_B:=\frac{C_\Gamma}{c}(C^*+C_0) 
\end{eqnarray}
we have
\begin{eqnarray}\label{so2}
|| \partial_t \eta^{n+1}(t)-v^n||_{L^2(\Gamma)}^2 \leq 
\Big( 2\xi_k C_B+ 4C_R^2C_\Gamma \big(||\mathcal{F}(0)||_{H^{-a}}^2 + C_B\big) \sum_{i=1}^k \xi_i^{a/2}  \Big) (\Delta t)^2.
\end{eqnarray}

In order to make the right-hand side of $\eqref{so2}$ be bounded by $(\Delta t)^{\alpha}$ for $\alpha>0$, we impose the following condition for every step
\begin{equation}\label{key}
N(k) \geq T  \Big( 2\xi_k C_B+ 4C_R^2 C_\Gamma \big(||\mathcal{F}(0)||_{H^{-a}}^2 + C_B \big) \sum_{i=1}^k \xi_i^{a/2} \Big)^{\frac{1}{2-\alpha}} , ~~ 0< \alpha <2.
\end{equation}

Now we are ready to have:

\begin{lem}\label{lemmakey}
Assuming that $N(k)$ satisfies $\eqref{key}$,  we have the following boundedness:
\begin{enumerate}
\item[(1)] $||\partial_t \eta_{\Delta t,k}^{n+1}(t) - v_{\Delta
t,k}^n||_{L^2(\Gamma)}^2 \leq (\Delta t)^{\alpha}$, for all $t \in [n\Delta t,(n+1) \Delta t]$;
\item[(2)] $\partial_t \eta_{\Delta t,k}$ is  bounded in $L^\infty(0,T; L^2(\Gamma))$ uniformly with respect to $\Delta t,k$;
\item[(3)] $\nabla^{\tilde{\eta}_{\Delta t,k}}u_{\Delta t,k}$ is bounded in $L^2(0,T; L^2(\Omega))$ uniformly  with respect to $\Delta t,k$;
\item[(4)] $v_{\Delta t,k}$ is bounded in $L^2(0,T; H^{s}(\Omega))$ uniformly  with respect to $\Delta t,k$, for any $0<s<1/2$;
\item[(5)] $||\widetilde{\partial}_t \eta_{\Delta t,k} - \partial_t \eta_{\Delta t,k}||_{L^\infty(0,T; L^2(\Gamma))}^2 \leq C (\Delta t)^{2\alpha}$;
\item[(6)] $\sum_{n=0}^{N-1} ||\eta_{\Delta t,k}^{n+1}-v_{\Delta t,k}^{n+1}||_{L^2(\Gamma)}^2 \leq C$ and $\sum_{n=1}^{N-1} ||v_{\Delta t,k}^{n+1}-v_{\Delta t,k}^{n}||_{L^2(\Gamma)}^2\leq C$.
\end{enumerate}
\end{lem}
\begin{proof}

From $\eqref{so2}$ and $\eqref{key}$, the above first and second estimates follow immediately from the uniform bound of $v_{\Delta t,k}$ given in Lemma \ref{111}(4).

The third result is obtained in the same way as in \cite[Proposition 5.3] {multilayer}, by mapping the gradient back to the moving domain $\Omega^{\tilde{\eta}}$ and applying the transformed Korn\rq{}s inequality. We then use uniform bounds for $\partial_t \eta_{\Delta t,k}$ in $L^\infty(0,T; L^2(\Gamma))$ and $\eta_{\Delta t,k}$ in $L^\infty(0,T; H_0^2(\Gamma))$ to obtain the uniform Korn\rq{}s constant, and since $\sum_{i=1}^N D_{\Delta t, k}^n$ is uniformly bounded from Lemma \ref{111}(6), the boundedness in the statement $(3)$ follows. Now, since the function $v_{\Delta t,k}$ is the trace of $\hat{\bb{u}}_{\Delta t,k}$, it enjoys a certain trace regularity (see \cite{Boris}). In particular
\begin{eqnarray}\label{tracevu}
||v_{\Delta t,k}^n||_{H^s(\Gamma)} \leq C||\hat{\bb{u}}_{\Delta t,k}^n||_{H^1(\Omega^{\tilde{\eta}_{\Delta t,k}^{n-1}})}
\end{eqnarray}
for all $0\leq n \leq N$, where $C$ depends on $s$ and $||\eta_{\Delta t,k}^{n-1}||_{C^{0,2s}(\Gamma)}$ which is uniformly bounded due to Lemma $\ref{111}(2)$, so the constant $C$ is uniform with respect to $n,\Delta t,k$. Now, the statement $(4)$ follows by $\eqref{tracevu}$ and the third statement.

Now, to estimate the difference between $\widetilde{\partial}_t \eta^{n}$ and $\partial_t \eta^n$, we first take the difference between equation $(2.1)$ for $t$ and $\tau-$moments in $[n\Delta t,(n+1)\Delta t]$, then choose $\psi = w_i$ and sum over $i=1,...,k$. Following the calculation in $\eqref{so4}$ and $\eqref{so3}$, we have
\begin{equation}\label{boundpartial}
\frac{1}{(\Delta t)^2}||\partial_t \eta^n(t)-\partial_t \eta^n(\tau)||_{L^2(\Gamma)}^2  \leq 2\xi_k|| \eta^n(t)-\eta^n (\tau)||_{H^2(\Gamma)}^2 + 2C_{R}^2||\eta^n(t)-\eta^n (\tau)||_{H^2}^2\sum_{i=1}^k||w_i||_{H^a}^2,
\end{equation}
and from $\eqref{so}$ and $\eqref{so8}$,
\begin{eqnarray}
\frac{1}{(\Delta t)^2}||\partial_t \eta^n(t)-\partial_t \eta^n(\tau)||_{L^2(\Gamma)}^2  &\leq& 2\xi_k C_\Gamma  ||\eta^n(t)-\eta^n(\tau)||_{L^2(\Gamma)}^2 \Big(  \xi_k +C_\Gamma C_R^2  \sum_{i=1}^k \xi_i^{a/2}  \Big) \nonumber \\
&\leq& 2\xi_k C_\Gamma  ||\int_{\tau}^{t} \partial_t \eta^{n}||_{L^2(\Gamma)}^2 \Big(  \xi_k +C_\Gamma C_R^2  \sum_{i=1}^k \xi_i^{a/2}  \Big)  \nonumber\\
&\leq& 2\xi_k C_\Gamma  || \partial_t \eta^{n}||_{L^\infty(n\Delta t,(n+1)\Delta t; L^2(\Gamma_)}^2 |t-\tau|^2  \Big(  \xi_k +C_\Gamma C_R^2  \sum_{i=1}^k \xi_i^{a/2}  \Big)   \nonumber\\
&\leq& 8\xi_k C_B |t-\tau|^2 \Big(  \xi_k +C_\Gamma C_R^2  \sum_{i=1}^k \xi_i^{a/2}  \Big).      \label{so7}
\end{eqnarray}
The last inequality here can be obtained by choosing, for example, $(\Delta t)^{\alpha} \leq \frac{1}{c}(C^*+C_0)$, and then by Lemma $\ref{111}(4)$, statement $(1)$ and $\eqref{CB}$, we bound
\begin{eqnarray*}
||\partial_t \eta^n||_{L^\infty((n-1)\Delta t, n\Delta t; L^2(\Gamma))}^2 \leq 2||\partial_t \eta^{n}(t) - v^{n-1}||_{L^\infty((n-1)\Delta t, n\Delta t;L^2(\Gamma))}^2+ 2||v^{n-1}||_{L^2(\Gamma)}^2 \leq \frac{4}{c}(C^*+C_0) = \frac{4C_B}{C_\Gamma}.
\end{eqnarray*}
Next, by the  H\"{o}lder inequality we have
\begin{eqnarray}
|  \widetilde{\partial}_t \eta^{n+1} - \partial_t \eta^{n+1}(\tau)|^2 &=& \Big|  \bint_{n\Delta t}^{(n+1)\Delta t} \ddfrac{\partial_t \eta^{n+1}(t)-\partial_t \eta^{n+1}(\tau)}{\Delta t} dt \Big|^2 \nonumber \\
 &\leq& \frac{1}{\Delta t}\bint_{n\Delta t}^{(n+1)\Delta t} \Big| 
\partial_t \eta^{n+1}(t)-\partial_t \eta^{n+1}(\tau) \Big|^2 dt \label{justsomeineq}
\end{eqnarray}
Now by integrating this inequality over $\Gamma$ and using the inequality $\eqref{so7}$, we obtain
\begin{equation}\label{inq}
\ddfrac{1}{\Delta t}||\widetilde{\partial}_t \eta^n - \partial_t \eta^n(\tau)||_{L^2(\Gamma)}^2 \leq 8 \xi_k C_B \Big(  \xi_k +C_\Gamma C_R^2  \sum_{i=1}^k \xi_i^{a/2}  \Big)  \int_{n\Delta t}^{(n+1)\Delta t} |\tau - t|^2 dt
\end{equation} 
From $\eqref{key}$, we can estimate 
\begin{eqnarray*}
(\Delta t)^{\alpha-2}\geq 2 \xi_k C_B\quad \text{and} \quad \frac{1}{2C_B} (\Delta t)^{\alpha-2} \geq \Big(  \xi_k +C_\Gamma C_R^2  \sum_{i=1}^k \xi_i^{a/2}  \Big)
\end{eqnarray*}
 which gives us
\begin{eqnarray}\label{anothereq}
\frac{4}{C_B} (\Delta t)^{2\alpha-4} \geq   8\xi_k C_B \Big(  \xi_k +C_\Gamma C_R^2  \sum_{i=1}^k \xi_i^{a/2}  \Big)
\end{eqnarray}
Now, combining $\eqref{inq}$ and $\eqref{anothereq}$, we conclude:
\begin{eqnarray*}
||\widetilde{\partial}_t \eta_{\Delta t,k}^n - \partial_t \eta_{\Delta t,k}^n||_{L^\infty((n-1)\Delta t, n\Delta t; L^2(\Gamma))}^2 \leq \frac{ 4(\Delta t)^{2\alpha}}{3C_B},
\end{eqnarray*}
and since $n$ was arbitrary, the statement $(5)$ follows. 

Finally, from Lemma \ref{111}(6) and the statements $(1)$ and $(5)$, we obtain:
\begin{eqnarray*}
\sum_{n=0}^{N-1} ||\eta_{\Delta t,k}^{n+1}-v_{\Delta t,k}^{n+1}||_{L^2(\Gamma)}^2  \\
\leq \sum_{n=0}^{N-1} ||v_{\Delta t,k}^{n+1}-\widetilde{\partial}_t \eta_{\Delta t,k}^{n+1}||_{L^2(\Gamma)}^2 + \sum_{n=0}^{N-1} ||\widetilde{\partial}_t\eta_{\Delta t,k}^{n+1}-\partial_t\eta_{\Delta t,k}^{n+1}||_{L^2(\Gamma)}^2 \leq C
\end{eqnarray*}
and consequently
\begin{eqnarray*}
\sum_{n=1}^{N-1} ||v_{\Delta t,k}^{n+1}-v_{\Delta t,k}^{n}||_{L^2(\Gamma)}^2 \nonumber \\
\leq \sum_{n=1}^{N-1} ||v_{\Delta t,k}^{n+1}-\partial_t \eta_{\Delta t,k}^{n+1}||_{L^2(\Gamma)}^2 +\sum_{n=1}^{N-1} ||\partial_t\eta_{\Delta t,k}^{n+1}-v_{\Delta t,k}^{n}||_{L^2(\Gamma)}^2 \leq C,
\end{eqnarray*}
so the proof is complete.

\end{proof}

\section{Convergence of the approximate solutions}
In this section, by using the estimates from section 3, we will first obtain the weak convergence of the approximate solutions. After that, by using the Theorem \cite[Theorem 3.1]{newcompact} and some standard compactness results, we will deduce the strong convergence as well.
\subsection{Weak convergence}
In order to prove the weak convergence of approximate solutions, we first need to prove that LE mapping defined in section 2.3 is well-defined independently from the choice of $\Delta t$ and $k$ on the
whole time interval $[0,T]$:

\begin{lem}\label{112}
Assuming that $N(k)$ satisfies $\eqref{key}$, we can choose  $T>0$ independent of $\Delta t$ and $k$ such that for all $t \in [0,T]$ 
\begin{eqnarray*}
0<C_{min}\leq 1+\eta_{\Delta t,k}(t,X)\leq C_{max}
\end{eqnarray*}

\end{lem}
\begin{proof}
Using the uniform estimate of $\partial_t \eta_{\Delta t,k}$ given in Lemma \ref{111}(3), we first obtain:
\begin{eqnarray*}
||\eta_{\Delta t,k}(t)-\eta_{\Delta t,k}(0)||_{L^2(\Gamma)} \leq
||\int_{0}^{t} \partial_t \eta_{\Delta t,k}(h) dh||_{L^2(\Gamma)}
\leq CT,
\end{eqnarray*}
and from the uniform estimate of $\eta_{\Delta t,k}$ in
$L^\infty(0,T;H_0^2(\Gamma))$ given in Lemma \ref{111}(2), we have:
\begin{eqnarray*}
||\eta_{\Delta t,k}(t)-\eta_{\Delta t,k}(0)||_{H_0^2(\Gamma)}
\leq 2C.
\end{eqnarray*}
By the interpolation inequality for Sobolev spaces we have:
\begin{eqnarray*}
||\eta_{\Delta t,k}(t)-\eta_{\Delta t,k}(0)||_{H^{3/2}(\Gamma)}
\leq 2CT^{1/4}.
\end{eqnarray*}
Since $H^{3/2}(\Gamma)$ is imbedded into
$C(\overline{\Gamma})$, we can bound
\begin{eqnarray*}
\eta_0+1 + 2CT^{1/4} \geq \eta_{\Delta t,k}(t)+1 \geq \eta_0+1- 2CT^{1/4}.
\end{eqnarray*}
Since $\eta_0+1 \geq C >0$, we can take $T$ small enough such that 
\begin{eqnarray*}
\eta_{\Delta t,k}+1 \geq C-2CT^{1/4} \geq C_{min} >0,
\end{eqnarray*}
and now the upper bound $C_{max}$ follows easily. This finishes the proof.
\end{proof}

Notice that now on the time interval $[0,T]$ the Jacobian of the LE mapping $A_{\eta}(t)$ is uniformly bounded from above and below by two positive constants. 

The following result is a direct consequence of boundedness given in Lemma $\ref{111}$:

\begin{lem}\label{weakconv}
Assuming that $N(k)$ satisfies $\eqref{key}$, there exist subsequences of $(\mathbf{u}_{\Delta t,k})_{\Delta t,k},(v_{\Delta
t,k})_{\Delta t,k}$ and $(\eta_{\Delta t,k})_{\Delta t,k}$, and the
functions $v \in L^\infty(0,T; L^2(\Omega))$, $\eta \in
W^{1, \infty}(0,T; L^2(\Gamma))\cap L^\infty(0,T;
H_0^2(\Gamma))$ and $\mathbf{u} \in L^\infty(0,T; L^2(\Omega))\cap
L^2(0,T; H^1(\Omega))$ such that:
\begin{eqnarray*}
\eta_{\Delta t,k}& \rightharpoonup &\eta \text{ weakly* in
}L^\infty(0,T; H_0^2(\Gamma)), \\
\partial_t \eta_{\Delta t,k} &\rightharpoonup& \partial_t \eta
\text{ weakly* in }L^\infty(0,T; L^2(\Gamma)) \\
v_{\Delta t,k} &\rightharpoonup &v \text{ weakly* in }
L^\infty(0,T; L^2(\Gamma)), \\
\mathbf{u}_{\Delta t,k} &\rightharpoonup &\mathbf{u} \text{ weakly* in
}L^\infty(0,T; L^2(\Omega)), \\
\nabla^{\tilde{\eta}^{n+1}}\mathbf{u}_{\Delta t,k} &\rightharpoonup & M \text{ weakly in }L^2(0,T;
L^2(\Omega)), 
\end{eqnarray*}
as $k \to +\infty$.
\end{lem}

We will later prove that $M$ is equal to $\nabla^\eta \bb{u}$. 

\subsection{Strong convergence}
In this section, we will prove the strong convergence of the approximate functions, which is essential in passing the limit in the nonlinear terms of the approximate problem. The key part is to prove the strong convergence of $\bb{u}_{\Delta t,k}$ and $v_{\Delta t,k}$ in $L^2(0,T; L^2(\Omega))$ and $L^2(0,T; L^2(\Gamma))$, respectively, which will be done by using the recent abstract result \cite[Theorem 3.1]{newcompact}. Troughout this section, we will assume that $T/\dt = N(k)$ satisfies the condition $\eqref{key}$.

\subsubsection{The strong convergence of $\bb{u}_{\Delta t,k}$ and $v_{\Delta t,k}$}
We shall apply \cite[Theorem 3.1]{newcompact} to study the sequence $(\hat{\bb{u}}_{\Delta t, k}, v_{\Delta t,k})_{k}$ (recall that $\Delta t = \Delta t(k)$ and that $\hat{\bb{u}}_{\Delta t, k}^{n} = \bb{u}_{\Delta t, k}^n \circ A_{\Delta t,k}^{n-1}$). First, we start by introducing the following spaces on the moving domain, the following test function space
\begin{eqnarray*}
Q_{\Delta t,k}^n = \{(\hat{\bb{q}},\psi) \in H^4(\Omega^{\eta_{\Delta t,k}^{n}}) \times H_{0,k}^2(\Gamma): \nabla \cdot \hat{\bb{q}} =0,~ \hat{\bb{q}}(t,X, \eta_{\Delta t,k}^n(X)) = \psi(t,X) \bb{e}_3 \}
\end{eqnarray*}
which is dense in the space of functions $(\bb{q} \circ (A_{\Delta t,k}^n)^{-1},\psi)$ such that $(\bb{q},\psi) \in \mathcal{W}_k^{\tilde{\eta}_{\Delta t,k}^{n}}$, and the following solution space
\begin{eqnarray*}
V_{\Delta t,k}^{n} =  \{(\hat{\bb{u}},v) \in H^{1}(\Omega^{\tilde{\eta}_{\Delta t,k}^{n}}) \times H_k^{\theta}(\Gamma): \nabla \cdot \hat{\bb{u}} =0,~ \hat{\bb{u}}(t,X, \widetilde{\eta}_{\Delta t,k}^{n}(X)) = v(t,X) \bb{e}_3 \}
\end{eqnarray*}
where $0<\theta<1/2$ is a parameter (which will be fixed from here onwards) and $H_k^{\theta}(\Gamma) = (\text{span}(\{w_i\}_{ 1\leq
i \leq k}), || \cdot||_{H^\theta(\Gamma)})$. Notice that if $(\bb{u}_{\Delta t,k}^{n+1}, v_{\Delta t,k}^{n+1}) \in \mathcal{W}_k^{\tilde{\eta}_{\Delta t,k}^n}$ then $(\hat{\bb{u}}_{\Delta t,k}^{n+1}, v_{\Delta t,k}^{n+1})\in V_{\Delta t,k}^{n}$ (due to the trace regularity, see Lemma $\ref{lemmakey}(4)$), so this choice of space is suitable for the solution on the moving domain.

We want to introduce function spaces which are defined on a maximal bounded domain containing all fluid domains $\Omega^{\eta_{\Delta t,k}^i}$ for $0\leq i \leq N-1$. For this reason, we define a smooth scalar function $M(X)$ on $\Gamma$ such that 
\begin{eqnarray}\label{functionM}
\begin{rcases}
&M(X)\geq \eta_{\Delta t,k}^n(X),\quad \text{for all } X\in \Gamma, k \in \mathbb{N}, T/\Delta t(k)=N \in \mathbb{N}~ \text{and } 0\leq n\leq N-1, \\
&|M(X) - \eta_0(X)|_{\infty}\leq C(T), \quad \text{where } C(T) \to 0, \text{when } T \to 0 
\end{rcases}
\end{eqnarray}
The existence of such a function follows by the fact that $\eta_{\Delta t,k}$ is uniformly bounded in $W^{1,\infty}(0,T; L^2(\Gamma))\cap L^{\infty}(0,T; H_0^2(\Gamma))$ (from Lemmas $\ref{111}$ and $\ref{lemmakey}$) which is imbedded into $C^{0,\alpha}(0,T; C^{0,1-2\alpha}(\Gamma))$, for any $0<\alpha<1/2$. Now we define the maximal domain 
\begin{eqnarray*}
\Omega^M = \{(X,z): X\in \Gamma,-1 < z < M(X)\}
\end{eqnarray*}
and the following Hilbert spaces:
\begin{eqnarray*}
&H = L^2(\Omega^M) \times L^2(\Gamma)& \\
&V = H^{\theta}(\Omega^M) \times H^{\theta}(\Gamma)&
\end{eqnarray*}
We want to prove the precompactness of $\{ (\hat{\bb{u}}_{\Delta t,k}, v_{\Delta t,k})\}_k$ in $L^2(0,T ;H)$, so we need the functions $\hat{\bb{u}}_{\Delta t, k}$ to be defined on $\Omega^M$. For that reason, we extend the functions $\hat{\bb{u}}_{\Delta t, k}$ by zero to $\Omega^M$, which we still denote as $\hat{\bb{u}}_{\Delta t, k}$ without any confusion. In Appendix B, we prove that these extended functions are indeed in the space $H^\theta(\Omega^M)$, and that this extension maps $(V_{\Delta t,k}^n, Q_{\Delta t,k}^n) \hookrightarrow V\times V$ continuously (uniformly in $n,\Delta t,k$).

In order to obtain this strong convergence, we will verify that the hypotheses in \cite[Theorem 3.1]{newcompact} hold:

\begin{lem}\label{assumab}
Let $(\bb{u}_{\Delta t,k}^{n}, v_{\Delta t,k}^n) \in \mathcal{W}_k^{\tilde{\eta}_{\Delta t,k}^{n-1}}$ be the solution of the (FSP) obtained in the Section 3 and $(\hat{\bb{u}}_{\Delta t,k}^n, v_{\Delta t,k}^n) \in V_{\Delta t,k}^{n-1} $ the corresponding solution on the moving domain. Then:
\begin{enumerate}
\item[(1)] $\sum_{n=1}^{N}||\hat{\bb{u}}_{\Delta t,k}^n ||_{H^1(\Omega^{\tilde{\eta}_{\Delta t,k}^{n-1}})}^2\Delta t \leq C$ and $\sum_{n=1}^{N}||\bb{v}_{\Delta t,k}^n ||_{H^\theta(\Gamma)}^2\Delta t \leq C$;
\item[(2)] $||\hat{\bb{u}}_{\Delta t,k}||_{L^\infty(0,T; L^2(\Omega^M))} \leq C$ and $||v_{\Delta t,k}||_{L^\infty(0,T; L^2(\Gamma))} \leq C$
\end{enumerate}
Moreover, we have
\begin{eqnarray}\label{assumb}
||P_{\Delta t,k}^n \frac{\hat{\bb{u}}_{\Delta t,k}^{n+1} - \hat{\bb{u}}_{\Delta t,k}^{n}}{\Delta t}||_{(Q_{\Delta t,k}^n)'} \leq C(|| \hat{\bb{u}}_{\Delta t,k}^{n+1}||_{V_{\Delta t,k}^n}+1)
\end{eqnarray}
where $P_{\Delta t,k}^n$ is the orthogonal projection onto $\overline{Q_{\Delta t,k}^n}^H$ and $(Q_{\Delta t,k}^n)'$ is the dual space of $Q_{\Delta t,k}^n$.
\end{lem}
\begin{proof}
Without any confusion, in this proof we will omit $\Delta t,k$ in the subscript of functions. First, the statement $(1)$ follows from Lemmas $\ref{111}(6)$ and $\ref{lemmakey}(4)$, while the statement $(2)$ follows from Lemma $\ref{111}(4)$ and $(5)$.

Next, to prove $\eqref{assumb}$, we want to estimate the following norm 
\begin{eqnarray*}
||P_{\Delta t,k}^n \frac{\hat{\bb{u}}^{n+1} - \hat{\bb{u}}^{n}}{\Delta t}||_{(Q_{\Delta t,k}^n)'} = \sup\limits_{||(\hat{\bb{q}},\psi)||_{Q_{\Delta t,k}^n}=1} \Big| \bint_{\Omega^{\eta^n}} \frac{\hat{\bb{u}}^{n+1} - \hat{\bb{u}}^{n}}{\Delta t} \cdot \hat{\bb{q}} + \bint_\Gamma \frac{v^{n+1} - v^{n}}{\Delta t} \cdot \psi  \Big|
\end{eqnarray*}
so we sum the sub-problem $\eqref{SSP}$ integrated over $(n\Delta t, (n+1)\Delta t)$ and divided by $\Delta t$ with the sub-problem $\eqref{FSP}$, and map it back to the moving domain $\Omega^{\tilde{\eta}^{n}}$ by using the mapping $(A^n)^{-1}$, to obtain
\begin{align}
&\ddfrac{1}{2} \bint_{\Omega^{\tilde{\eta}^{n}}}\big( ((\overline{\bb{u}}^{n} - \bb{w}^{n+1})\cdot
\nabla)\hat{\bb{u}}^{n+1} \cdot \hat{\bb{q}} - ((\overline{\bb{u}}^n-\bb{w}^{n+1} ) \cdot
\nabla) \hat{\bb{q}} \cdot \hat{\bb{u}}^{n+1} \big) &\nonumber \\
&+\bint_{\Omega^{\tilde{\eta}^{n}}}\ddfrac{\hat{\bb{u}}^{n+1}-\overline{\bb{u}}^n}{\Delta t} \cdot \hat{\bb{q}} +
\frac{1}{2}\bint_{\Omega^{\tilde{\eta}^{n}}} \ddfrac{J^{n+1}-J^n}{J^n \Delta t} \hat{\bb{u}}^{n+1}
\cdot \hat{\bb{q}}&  \nonumber  \\
&+2 \mu \int_{\Omega^{\tilde{\eta}^{n}}} \bb{D} (\hat{\bb{u}}^{n+1}):\bb{D}(\hat{\bb{q}}) + \bint_\Gamma
\ddfrac{v^{n+1}-v^n}{\Delta t} \psi \nonumber \\
&+\frac{1}{\Delta t}\int_{n\Delta t}^{(n+1)\Delta t}\int_\Gamma \Delta \eta^{n+1} \Delta \psi +
\frac{1}{\Delta t}\int_{n\Delta t}^{(n+1)\Delta t}\int_\Gamma \mathcal{F}(\eta^{n+1}(t)) \psi = 0, ~~  \label{summedupmapped}
\end{align}
for all $(\hat{\bb{q}},\psi) \in Q_{\Delta t,k}^n$, where $\overline{\bb{u}}^n = \bb{u}^n \circ (A_{\Delta t,k}^n)^{-1}$. 
Notice that in one term $1/J^n$ appeared (and in others got canceled with $J^n$) since its the Jacobian of the mapping $(A^n)^{-1}$. We kept the notation $J^n,J^{n+1}$ and $\bb{w}^{n+1}$ even though these functions are defined on $\Omega$ since these functions do not depend on the coordinate $z$, so there is no confusion. First, by the triangle inequality
\begin{eqnarray*}
 \Big| \bint_{\Omega^{\tilde{\eta}^{n}}} \frac{\hat{\bb{u}}^{n+1} - \hat{\bb{u}}^{n}}{\Delta t} \cdot \hat{\bb{q}} + \bint_\Gamma \frac{v^{n+1} - v^{n}}{\Delta t} \cdot \psi\Big| \\   \leq  \Big| \bint_{\Omega^{\tilde{\eta}^{n}}} \frac{\hat{\bb{u}}^{n+1} - \overline{\bb{u}}^{n}}{\Delta t} \cdot \hat{\bb{q}} + \bint_\Gamma \frac{v^{n+1} - v^{n}}{\Delta t} \cdot \psi  \Big| 
+ \Big| \bint_{\Omega^{\tilde{\eta}^{n}}} \frac{\hat{\bb{u}}^{n} - \overline{\bb{u}}^{n}}{\Delta t} \cdot \hat{\bb{q}}  \Big|
\end{eqnarray*}
The first term on the right-hand side is estimated from the equation $\eqref{summedupmapped}$
in the following way:
\begin{eqnarray*}
 &&\Big| \bint_{\Omega^{\eta_{\Delta t,k}^n}} \frac{\hat{\bb{u}}^{n+1} - \overline{\bb{u}}^{n}}{\Delta t} \cdot \hat{\bb{q}} + \bint_\Gamma \frac{v^{n+1} - v^{n}}{\Delta t} \cdot \psi  \Big| \\ &\leq& \big( ||\overline{\bb{u}}^n||_{L^2(\Omega^{\tilde{\eta}^{n}})}+|| \bb{w}^{n+1}||_{L^2(\Omega^{\tilde{\eta}^{n}})} \big) ||\nabla \hat{\bb{u}}^{n+1}||_{L^2(\Omega^{\tilde{\eta}^{n}})} || \hat{\bb{q}}||_{L^\infty(\Omega^{\tilde{\eta}^{n}})}\\
&+& \big( ||\overline{\bb{u}}^n||_{L^2(\Omega^{\tilde{\eta}^{n}})}+|| \bb{w}^{n+1}||_{L^2(\Omega^{\tilde{\eta}^{n}})} \big) ||\nabla  \hat{\bb{q}}||_{L^\infty(\Omega^{\tilde{\eta}^{n}})} ||\hat{\bb{u}}^{n+1}||_{L^2(\Omega^{\tilde{\eta}^{n}})} +C|| \widetilde{\partial}_t \eta^{n+1} ||_{L^2(\Gamma)}
||\hat{\bb{u}}_{\Delta t,k}^{n+1}||_{L^2(\Omega^{\tilde{\eta}^{n}})} || \hat{\bb{q}}||_{L^\infty(\Omega^{\tilde{\eta}^{n}})}\\
 &+& C || \nabla \hat{\bb{u}}_{\Delta t,k}^{n+1}||_{L^2((\Omega^{\tilde{\eta}^{n}})}|| \nabla \hat{\bb{q}}||_{L^2(\Omega^{\tilde{\eta}^{n}})} + ||\Delta \eta||_{L^\infty(n\Delta t, (n+1)\Delta t; L^2(\Gamma))} ||\Delta \psi||_{L^2(\Gamma)}\\
 &+& C( || \mathcal{F}(0)||_{H^{-2}(\Gamma)} + ||\eta||_{L^\infty(n\Delta t, (n+1)\Delta t; H^{2-\varepsilon}(\Gamma))}) ||\psi ||_{H^2(\Gamma)} \leq C(|| \hat{\bb{u}}^{n+1} ||_{V_{\Delta t,k}^n}+1) ||(\hat{\bb{q}},\psi)||_{Q_{\Delta t,k}^n}
\end{eqnarray*}
where we used the estimates from Lemmas $\ref{111}$ and $\ref{lemmakey}$, the assumption $(A1)$ for the nonlinear function $\mathcal{F}$ and the Sobolev imbeddings. To estimate the next term, we decompose $\Omega^{\tilde{\eta}^n} =S_1 \cup S_2$, where $S_1 = \Omega^{\tilde{\eta}^n} \cap \Omega^{\tilde{\eta}^{n-1}}$, $S_2 = \Omega^{\tilde{\eta}^n} \setminus \Omega^{\tilde{\eta}^{n-1}}$ and bound
\begin{eqnarray*}
\Big| \bint_{S_1} \frac{\hat{\bb{u}}^{n} - \overline{\bb{u}}^{n}}{\Delta t} \cdot \hat{\bb{q}}  \Big| = \frac{1}{\Delta t}\Big| \bint_{S_1} \hat{\bb{u}}^{n}(X,z) - \hat{\bb{u}}^{n}(X,\frac{(\eta^{n-1} +1)(z+1)}{\eta^n+1}-1) \cdot \hat{\bb{q}}  \Big| \\
\leq C ||\nabla \hat{\bb{u}}^n ||_{L^2(S_1)} || \widetilde{\eta}^{n}||_{L^2(\Gamma)} ||\hat{\bb{q}}||_{L^\infty(S_1)}
\end{eqnarray*}
by the mean-value Theorem. Next,
\begin{eqnarray*}
\Big| \bint_{S_2} \frac{\hat{\bb{u}}^{n} - \overline{\bb{u}}^{n}}{\Delta t} \cdot \hat{\bb{q}}  \Big| \leq \frac{1}{\Delta t}\bint_{\Gamma} \bint_{\tilde{\eta}^{n-1}}^{\tilde{\eta}^{n}} |\overline{\bb{u}}^{n}| |\hat{\bb{q}}|  dz dX \leq ||\hat{\bb{u}}^n ||_{L^2(S_2)}||\widetilde{\partial}_t\eta^n||_{L^2(\Gamma)}||\hat{\bb{q}}||_{L^\infty(S_2)}
\end{eqnarray*}
since $\hat{\bb{u}}_{\Delta t,k}^{n}=0$ on $S_2$. Combining the previous four inequalities, by using the imbedding of $H^4(\Omega^{\tilde{\eta}^n})$ into $W^{1,\infty}(\Omega^{\tilde{\eta}^n})$, the uniform bounds from Lemmas $\ref{111}$ and $\ref{lemmakey}$, the inequality $\eqref{assumb}$ follows.
\end{proof}

\begin{cor}\label{comcol}
Assuming that $N(k)$ satisfies $\eqref{key}$, then $\bb{u}_{\Delta t,k} \to \bb{u}$ in $L^2(0,T; L^2(\Omega))$ and $v_{\Delta t,k} \to v$ in $L^2(0,T; L^2(\Gamma))$ as $k\to \infty$.
\end{cor}
\begin{proof}
From Lemma $\ref{assumab}$, we have that the assumptions $A$ and $B$ given in \cite[Theorem 3.1]{newcompact} are satisfied. Notice that the assumption $A3$ given in \cite[Theorem 3.1]{newcompact} is not necessary to verify due to \cite[Theorem 3.2]{newcompact}.
 
To prove the assumption $C$ given in \cite[Theorem 3.1]{newcompact}, one can use almost the same construction as in \cite[Example 4.2]{newcompact} with the following changes. First, since $\eta_{\Delta t,k}$ is only uniformly bounded in $C^{0,\alpha}(0,T; C^{0,1-2\alpha}(\Gamma))$ for $0<\alpha<1/2$, the functions constructed in \cite[Lemma 4.5]{newcompact} would satisfy a weaker inequality - on the right-hand side of the third inequality, we would have $C(l\Delta t)^\alpha$, instead of $C\sqrt{l\Delta t}$. Consequently, the corresponding operator $I_{n,l,\Delta t}^i$ would satisfy a weaker assumption $C1$ given in \cite[Theorem 3.1]{newcompact}, where the right-hand side of the inequality $(3.3)$ is replaced by $C(l\Delta t)^\alpha$. Nevertheless, one can carry out the proof of \cite[Theorem 3.1]{newcompact} in the same way with this weaker assumption, by proving the \cite[Lemma 3.1]{newcompact} where the right-hand side of the inequality is replaced by $C(l\Delta t)^\alpha$, and the \cite[Lemma 3.2]{newcompact} where the last term on the right-hand side of the inequality is replaced by $C(\delta)C(l\Delta t)^\alpha$. Then, one can conclude that this slightly modified \cite[Theorem 3.1]{newcompact} holds from the inequality $(3.23)$ in which the last term is replaced by $C(\delta)C(l\Delta t)^\alpha$. Notice that in our case we would choose $g(h)=Ch^{\alpha/2}$.

Therefore, we apply this version of \cite[Theorem 3.1]{newcompact} for the sequence $(\hat{\bb{u}}_{\Delta t,k},v_{\Delta t,k})_k$ and obtain that $\hat{\bb{u}}_{\Delta t,k} \to \hat{\bb{u}}$ in $L^2(0,T; L^2(\Omega^M))$ and $v_{\Delta t,k} \to v$ in $L^2(0,T; L^2(\Gamma))$. It remains to to prove that $\bb{u}_{\Delta t,k} \to \bb{u}$ in $L^2(0,T; L^2(\Omega))$, where $\bb{u} = \hat{\bb{u}}\circ A_\eta^{-1}$. Let $\phi \in L^2(\Omega)$ and
\begin{eqnarray*}
\hat{\phi}(t,X,z) = \begin{cases}
\phi \circ (A_{\Delta t,k}^n)^{-1},& \quad t \in [n\Delta t, (n+1)\Delta t), (X,z)\in \Omega^{\tilde{\eta}_{\Delta t,k}^{n}} \\
0,& \text{elsewhere in }\Omega^M
\end{cases}
\end{eqnarray*}
Now, 
\begin{eqnarray*}
\lim\limits_{k\to \infty} \int_0^T\int_\Omega \bb{u}_{\Delta t,k} \phi = \lim\limits_{k\to \infty}\int_0^T\int_{\Omega^M} \frac{1}{T_{\Delta t} J_{\Delta t,k}} \hat{\bb{u}}_{\Delta t,k} \hat{\phi} =\int_0^T\int_{\Omega^M} \frac{1}{J} \hat{\bb{u}} \hat{\phi} = \int_0^T\int_\Omega \bb{u} \phi 
\end{eqnarray*}
where $T_{\Delta t} J_{\Delta t,k}(t) = J_{\Delta t,k}(t-\Delta t)$, and since $T_{\Delta t} J_{\Delta t,k},J_{\Delta t,k} \to J$ in $C([0,T] \times \overline{\Omega})$ (see Corollary $\ref{conv}$), we have that $\bb{u}_{\Delta t,k} \rightharpoonup \bb{u}$ in $L^2(0,T; L^2(\Omega))$. Also,
\begin{eqnarray*}
\lim\limits_{k\to \infty}||\bb{u}_{\Delta t,k}||_{L^2(0,T; L^2(\Omega))}^2&=&\lim\limits_{k\to \infty}\int_{0}^T \int_{\Omega} | \bb{u}_{\Delta t,k}|^2 =\lim\limits_{k\to \infty} \int_{0}^T \int_{\Omega^M} \frac{1}{J_{\Delta t,k}} |  \hat{\bb{u}}_{\Delta t,k}|^2 \\
&=& \int_{0}^T \int_{\Omega^M} \frac{1}{J} |  \hat{\bb{u}}|^2 =\int_{0}^T \int_{\Omega} | \bb{u}|^2= || \bb{u}||_{L^2(0,T; L^2(\Omega))}^2.
\end{eqnarray*}
Now, the weak convergence of $\bb{u}_{\Delta t,k}$ and the strong convergence of $||\bb{u}_{\Delta t,k}||_{L^2(0,T; L^2(\Omega))}^2$ imply that $\bb{u}_{\Delta t,k} \to \bb{u}$ in $L^2(0,T; L^2(\Omega))$, so we finish the proof.
\end{proof}

\subsubsection{Strong convergence of the approximate plate displacement}
First, since $\eta_{\Delta t,k}$ is uniformly bounded in $W^{1, \infty}(0,T; L^2(\Gamma))\cap L^\infty(0,T;H_0^2(\Gamma))$, from the continuous imbedding
\begin{eqnarray*}
W^{1, \infty}(0,T; L^2(\Gamma))\cap L^\infty(0,T; H_0^2(\Gamma))
\hookrightarrow C^{0, 1- \alpha}([0,T], H^{2\alpha}(\Gamma)), 0<
\alpha<1,
\end{eqnarray*}
we obtain uniform boundedness of $\eta_{\Delta t,k}$ in $C^{0, 1-
\alpha}([0,T], H^{2\alpha}(\Gamma))$. Since $H^{2\alpha}$ is compactly
imbedded into $H^{2\alpha-\epsilon}$ for any fixed $\epsilon >0$, and since the functions in
$C^{0,1-\alpha}(0,T; H^{2\alpha}(\Gamma))$ are uniformly
continuous in time on finite interval, by using the Arzela-Ascoli Theorem
we obtain the compactness in time as well. Therefore, we have
\begin{eqnarray*}
\eta_{\Delta t,k} \to \eta \text{ in } C([0,T]; H^{2s}(\Gamma)), 0<s<1,
\end{eqnarray*}
and
\begin{eqnarray*}
T_{\Delta t} \eta_{\Delta t,k} \to \eta \text{ in } C([0,T]; H^{2s}(\Gamma)), 0<s<1,
\end{eqnarray*}
as $\Delta t \to 0$. By the compactness we just obtained for $\eta$, since $\mathcal{F}$ is Lipschitz continuous from $H^{2-\epsilon}(\Gamma)$ to $H^{-2}(\Gamma)$ (see \eqref{assumption0}) and since constant $C_R$ is now uniformly bounded due to the uniform energy estimate, we have
\begin{eqnarray*}
||\mathcal{F}(\eta_{\Delta t,k}(t))-\mathcal{F}(\eta(t))||_{H^{-2}(\Gamma)}\leq C||\eta_{\Delta t,k}(t)-\eta(t)||_{H^{2-\epsilon}(\Gamma)},
\end{eqnarray*}
so $\mathcal{F}(\eta_{\Delta t,k}(t))\to \mathcal{F}(\eta(t))$ in $H^{-2}(\Gamma)$, for all $t \in [0,T]$. Next, from Lemma $\ref{lemmakey}(5),(6)$ and Corollary $\ref{comcol}$, when $N(k)$ satisfies $\eqref{key}$, we have that $\partial_t \eta_{\Delta t,k},\widetilde{\partial}_t \eta_{\Delta t,k} \to v$ in $L^2(0,T; L^2(\Gamma))$ and $v = \partial_t \eta$.
\begin{cor}\label{conv}
Assuming that $N(k)$ satisfies $\eqref{key}$, we have the following convergences:
\begin{enumerate}
\item[(1)] $\eta_{\Delta t,k},T_{\Delta t} \eta_{\Delta t,k}  \to \eta$, in $C([0,T] \times \overline{\Gamma})$;
\item[(2)] $\widetilde{\eta}_{\Delta t,k} \to \eta$ in
$L^\infty(0,T;C(\overline{\Gamma}))$;
\item[(3)]  $J_{\Delta t,k }$, $T_{\Delta t}J_{\Delta t,k } \to  J$ and $A_{\eta_{\Delta t,k}} \to A$, in $C([0,T] \times \overline{\Omega})$
\item[(4)] $\mathbf{w}_{\Delta t, k}\to \mathbf{w}$ in $L^2(0,T; L^2(\Omega))$;
\item[(5)] $(\nabla A_{\Delta t, k})^{-1}\to (\nabla A)^{-1}$ in
$C(0,T; H^{2s}(\Omega))$ for $s<1/2$,
\end{enumerate}
as $k \to \infty$.
\end{cor}

\section{Proof of the main result}
In this section, we will first construct appropriate test functions in certain spaces that will converge to the test functions of the weak solution defined in Definition $\ref{weaksolution}$. After that, we will prove the convergence of approximate solutions term by term by using the convergence results obtained in section 4. At the end of the section, we will prove that the life span of the solution is either $+\infty$ or up to the moment of the free boundary touch the bottom.
Throughout this section we assume that $N(k)$ satisfies the condition $\eqref{key}$.

\subsection{Construction of test functions}
Now, for given test functions $(q,\psi) \in Q^\eta(0,T)$, where $\eta$ is the weak* limit of the sequence $\eta_{\Delta t,k}$ given in Lemma $\ref{111}$, we want to construct a sequence of test functions $\bb{q}_{\Delta t,k}$ and $\psi_{\Delta t,k}$ such that $(\bb{q}_{\Delta t,k}^n,\psi_{\Delta t,k}^n) \in \mathcal{W}_k^{\tilde{\eta}^{n-1}}$ for all $1 \leq n \leq N(k)$, and $(\bb{q}_{\Delta t,k},\psi_{\Delta t,k})\to(\bb{q},\psi)$ when $k\to \infty$ in a suitable space. 

Since the original problem and the original test functions are defined on the moving domain, we start by defining the uniform domains
\begin{eqnarray*}
\Omega_{max}= \bigcup_{\Delta t>0, n \in\mathbb{N}} \Omega^{\tilde{\eta}_{\Delta t,k}^n}, ~~\Omega_{min}= \bigcap_{\Delta t>0, n \in\mathbb{N}} \Omega^{\tilde{\eta}_{\Delta t,k}^n}
\end{eqnarray*}
We define $\chi_{max}(0,T)$ as the space of test functions $(\bb{r},\psi)$, where $\psi \in C_c^1([0,T); H_0^2(\Gamma))$ and $\bb{r} = \bb{r}_0+\bb{r}_1$ such that:
\begin{enumerate}
\item $\bb{r}_0$ is a smooth divergence free function with compact support in $\Omega^{\eta}(t)\cup \Sigma \cup \Gamma^\eta(t)$, extended by $0$ to $\Omega_{max} \setminus \Omega^{\eta}(t)$.
\item function
\begin{eqnarray}\label{ext}
\bb{r}_1= \begin{cases} 
(0,0,\psi)~~\text{ on } \Omega_{max}\setminus\Omega_{min} \\
\text{Divergence free extension to }\Omega_{min} ~ \text{(see }\cite[\text{p}. 127]{Galdi}).
\end{cases}
\end{eqnarray}
\end{enumerate} 
We then define
\begin{eqnarray*}
\chi^{\eta}(0,T)= \{ (\bb{q},\psi): \bb{q}(t ,X,z)=\bb{r} (t, X,z)_{|\Omega^{\eta}(t)}
\circ A_\eta(t), \text{ where } (\bb{r},\psi) \in \chi_{max}(0,T) \}.
\end{eqnarray*}
Notice that this function space is dense in $Q^\eta(0,T)$. \\

For approximate solution $X_{\Delta t,k}$ and given test functions $(\bb{q},\psi) \in \chi^{\eta}(0,T)$, with  
$$\bb{q}=\bb{r}_{|\Omega^{\eta}}\circ A_\eta, \quad \bb{r} = \bb{r}_0+ \bb{r}_1$$ 
being the decomposition we introduced,  define approximated test functions $(\bb{q}_{\Delta t,k}, \psi_{\Delta t,k})$ in the following way:
\begin{eqnarray*}
\psi_{\Delta t, k}(t, \cdot)=\psi_k(n\Delta t), ~~ t \in [n\Delta t, (n+1)\Delta t),
\end{eqnarray*}
where $\psi_k(t)$ is a projection of $\psi(t,\cdot )$ in $\mathcal{G}_k=\text{span}\{w_i : 1 \leq i \leq k \}$, and
\begin{eqnarray*}
\bb{q}_{\Delta t,k}(t, \cdot)= \bb{q}_{\Delta t,k}^{n}(t, \cdot) &:=&\bb{r}_k(n\Delta t,\cdot)_{|\Omega^{\tilde{\eta}_{\Delta t,k}^{n-1}}} \circ A_{\Delta t,k}^{n-1}, ~~t \in [n\Delta t, (n+1)\Delta t), \\
\bb{r}_k &=&\bb{r}_0 +\bb{r}_{1,k}
\end{eqnarray*}
where $\bb{r}_{1,k}$ is the extension of $(0,0,\psi_k)$ in the same way as given in $\eqref{ext}$.
Before we proceed to the convergence of approximate solutions, we first have some convergences for the test functions as follows: 

\begin{lem}\label{3.1}
For every $(\mathbf{q},\psi)\in \chi^\eta(0,T)$, we have:
\begin{eqnarray*} 
\bb{q}_{\Delta t,k} \to \bb{q}, ~~~\psi_{\Delta t, k} \to \psi, ~~~ \nabla \bb{q}_{\Delta t,k} \to \nabla \bb{q}, ~~ \text{uniformly on } [0,T] \times \Omega. 
\end{eqnarray*}
as $k\to \infty$.
\end{lem}
\begin{proof}
For $t \in [n\Delta t, (n+1)\Delta t)$,
\begin{align}
\bb{q}_{\Delta t,k}(t) - \bb{q}(t) &=\bb{r}_k\big(n\Delta t,A_{\Delta t,k}^{n-1}\big) - \bb{r}\big(t,A_\eta(t)\big) \nonumber \\
&=\Big(\bb{r}_k\big(n\Delta t,A_{\Delta t,k}^{n-1}\big) -\bb{r}\big(n\Delta t,A_{\Delta t,k}^{n-1}\big)\Big)
 +\bb{r}\big(n\Delta t,A_{\Delta t,k}^{n-1}\big)  - \bb{r}\big(t,A_\eta(t)\big) \label{testfunctions}
\end{align}
Since $\bb{r}_k \to\bb{r}$ when $k \to \infty$,  the first difference on the right hand side of  $\eqref{testfunctions}$ goes to $0$, so we only need to study the other difference. Decompose it into
\begin{eqnarray*}
\bb{r}\big(n\Delta t,A_{\Delta t,k}^{n-1}\big)  - \bb{r}\big(t,A_\eta(t)\big) = \bb{r}\big(n \Delta t,A_\eta(t)\big) - \bb{r}\big(t,A_\eta(t) \big) +\bb{r}\big(n\Delta t,A_{\Delta t,k}^{n-1}\big)  - \bb{r}\big(n \Delta t,A_\eta(n\Delta t)\big).
\end{eqnarray*}
By the mean-value Theorem, there is $\beta \in [0,1]$ such that
\begin{eqnarray*}
 \bb{r}\big(n \Delta t,A_\eta(t)\big) - \bb{r}\big(t,A_\eta(t)\big) = \partial_t \bb{r}\big(n\Delta t + \beta (t- n\Delta t), A_\eta(t)\big) (t- n\Delta t),
\end{eqnarray*}
which converges to $0$  as $k\to \infty$, and the remaining one
\begin{eqnarray*}
\bb{r}\big(n\Delta t,A_{\Delta t,k}^{n-1}\big)  - \bb{r}\big(n \Delta t,A_\eta(n\Delta t)\big) \to 0, ~~\text{as } k\to \infty
\end{eqnarray*}
from the strong convergence of $A_{\Delta t,k} \to A$. The same argument applies  to deduce $\psi_{\Delta t, k} \to \psi$ and $\nabla \bb{q}_{\Delta t,k} \to \nabla \bb{q}$ as $k\to \infty$.
\end{proof}

\begin{lem}\label{takoto}
We have:
\begin{eqnarray*}
\frac{\psi_{\Delta t,k }^{n+1}-\psi_{\Delta t,k }^{n}}{\Delta t}=: \text{d}\psi_{\Delta t,k }^n \to \partial_t \psi \text{ in } L^2(0,T; L^2(\Omega)), ~~ \text{as } k\to +\infty,
\end{eqnarray*}
and
\begin{eqnarray*}
\frac{\bb{q}_{\Delta t,k }^{n+1}-\bb{q}_{\Delta t,k }^{n}}{\Delta t}=: \text{d}\bb{q}_{\Delta t,k }^n \to \partial_t \bb{q} \text{ in } L^2(0,T; L^2(\Omega)), ~~ \text{as } k\to + \infty.
\end{eqnarray*}
\end{lem}

\begin{proof}
Obviously, we have
\begin{eqnarray*}
\frac{\psi_{\Delta t,k }^{n+1}-\psi_{\Delta t,k }^{n}}{\Delta t}= \frac{\psi_k((n+1)\Delta t)-\psi_{k }(n\Delta t)}{\Delta t} =\partial_t \psi(n\Delta t+\beta \Delta t) - \sum_{i=k+1}^\infty  \big(\partial_t \psi(n\Delta t+\beta \Delta t), w_i \big) w_i,
\end{eqnarray*}
for a $\beta \in (0,1)$. The last term goes to $0$ when $k \to +\infty$ by using the boundedness of $\partial_t  \psi$ in $L^\infty(0,T; L^2(\Gamma))$.

For the second convergence, we use similar argument as in the previous Lemma to get:
\begin{eqnarray*}
&\ddfrac{\bb{q}_{\Delta t,k }^{n+1}-\bb{q}_{\Delta t,k }^{n}}{\Delta t}= \ddfrac{1}{\Delta t} \Big( \bb{r}_k\big((n+1)\Delta t, A_{\Delta t,k}^{n}) - \bb{r}_k\big(n\Delta t, A_{\Delta t,k}^{n-1}\big)  \Big)& \\
&= \ddfrac{1}{\Delta t} \Big( \bb{r}_k\big((n+1)\Delta t, A_{\Delta t,k}^{n}\big) - \bb{r}_k\big((n+1)\Delta t, A_{\Delta t,k}^{n-1}\big) + \bb{r}_k\big((n+1)\Delta t, A_{\Delta t,k}^{n-1}\big) - \bb{r}_k\big(n\Delta t, A_{\Delta t,k}^{n-1}\big)  \Big) &\\
&=\nabla \bb{r}_k \big((n+1)\Delta t,\delta  \big)\ddfrac{A_{\Delta t,k}^{n}-A_{\Delta t,k}^{n-1}}{\Delta t} +\partial_t \bb{r}_k \big( (n+\beta)\Delta t , A_{\Delta t,k}^{n-1} \big),&
\end{eqnarray*}
with $\delta =A_{\Delta t,k}^{n-1}+ \gamma \big(A_{\Delta t,k}^{n}-A_{\Delta t,k}^{n-1}\big)$ for $\beta,\gamma \in [0,1]$. In the decomposition $\bb{r}_k = \bb{r}_0 +\bb{r}_{1,k}$, since term $\bb{r}_{1,k}$ is a simple extension of $\psi_k$, we have that $\nabla \bb{r}_{1,k}$ is the extension of $\nabla \psi_k$ and $\partial_t \bb{r}_{1,k}$ is the extension of $\partial_t \psi_k$. From the convergence $ \psi_k \to \psi$ in $C_c^1([0,T); H_0^2(\Gamma))$, we obtain $\nabla \bb{r}_k \to \nabla \bb{r}$ and $\partial_t \bb{r}_k \to \partial_t \bb{r}$ as $k\to +\infty$. By using the convergence of $A_{\Delta t,k} \to A_\eta$, the term $\frac{A_{\Delta t,k}^{n}-A_{\Delta t,k}^{n-1}}{\Delta t} =\bb{w}^{n}$ converges to $\bb{w}^\eta$ as $k\to +\infty$, we obtain
\begin{eqnarray*}
d \bb{q}_{\Delta t,k} \to \bb{w}^\eta \cdot \nabla \bb{r} + \partial_t \bb{r} = \partial_t \bb{q} ~~\text{ in } L^2(0,T ; L^2(\Omega)) \quad {\rm as}~k\to +\infty.
\end{eqnarray*}
\end{proof}

\subsection{Passage to the limit}

To define the approximate problem, for every $(\mathbf{q},\psi) \in \chi^\eta(0,T)$, by taking the sum $\eqref{SSP}$ and $\eqref{FSP}$ with test functions $\mathbf{q}_{\Delta t,k}$ and $\psi_{\Delta t,k}$, we get
\begin{align}
&\bint_0^T \Bigg[\ddfrac{1}{2} \bint_\Omega T_{\Delta t} J_{\Delta t, k} \big( ((T_{\Delta t}\mathbf{u}_{\Delta t,k} -\bb{w}_{\Delta t, k})\cdot
\nabla^{\eta_{\Delta t, k}})\mathbf{u}_{\Delta t, k} \cdot \mathbf{q}_{\Delta t, k} - ((T_{\Delta t}\mathbf{u}_{\Delta t,k}-\bb{w}_{\Delta t, k} ) \cdot
\nabla^{\eta_{\Delta t, k}}) \mathbf{q}_{\Delta t, k} \cdot \mathbf{u}_{\Delta t, k} \big)  \nonumber \\  &\nonumber \\
& +\bint_{\Omega} T_{\Delta t}J_{\Delta t, k} \partial_t\mathbf{u}_{\Delta t, k}^* \cdot \mathbf{q}_{\Delta t, k} +
\ddfrac{1}{2}\bint_\Omega \ddfrac{J_{\Delta t, k}-T_{\Delta t} J_{\Delta t,k}}{\Delta t} \mathbf{u}_{\Delta t, k}
\cdot \mathbf{q}_{\Delta t,k}+
(F(\eta_{\Delta t, k}), \psi_{\Delta t, k})+(\Delta \eta_{\Delta t, k}(t), \Delta \psi_{\Delta t, k}) \nonumber \\ & \nonumber \\
 &+2 \mu \bint_\Omega J^{n+1} \bb{D}^{\eta_{\Delta t, k}} (\mathbf{u}_{\Delta t, k}):\bb{D}^{\eta_{\Delta t, k}}(\mathbf{q}_{\Delta t, k})  
+\bint_\Gamma \partial_t v_{\Delta t, k}^* \cdot \psi_{\Delta t, k}  \Bigg] dt= 0, \label{thelasteq}
\end{align}
where $T_\dt f(t) = f(t-\Delta t)$ is the translation in time operator, and $\mathbf{u}_{\Delta t,k}^*$ and $v_{\Delta t,k}^*$ are the piece-wise linear approximations of $\mathbf{u}_{\Delta t,k}$ and $v_{\Delta t,k}$, i.e. for $t \in [n\Delta t, (n+1)\Delta t]$
\begin{eqnarray*}
\bb{u}_{\Delta t,k}^*(t) &=& \bb{u}_{\Delta t,k}^{n}+\ddfrac{\bb{u}_{\Delta t,k}^{n+1}-\bb{u}_{\Delta t,k}^{n}}{\Delta t} (t-n\Delta t), \\
v_{\Delta t,k}^*(t) &=& v_{\Delta t,k}^{n}+\ddfrac{v_{\Delta t,k}^{n+1}-v_{\Delta t,k}^{n}}{\Delta t} (t-n\Delta t).
\end{eqnarray*}
Notice that replacing the discretized time derivative with the time derivative of the piece-wise linear approximations does not change the value of the corresponding two integral terms. We rewrite the approximate problem $\eqref{thelasteq}$ as $\sum_{i=1}^8 I_i=0$, where $I_i$ represents each integral term on the left side of \eqref{thelasteq} for $1\le i\le 8$.

We still don\rq{}t know if the limit of $\nabla^{\tilde{\eta}_{\Delta t,k}} \mathbf{u}_{\Delta t,k}\to \nabla^\eta \mathbf{u}$. Since our space regularity on the fixed domain for $\mathbf{u}$ is less than $H^1$, this limit had to be passed on the moving domain. It was proved in \cite{multilayer,multilayer2} that:
\begin{eqnarray*}
\nabla^{\tilde{\eta}_{\Delta t,k}} \mathbf{u}_{\Delta t,k} \to \nabla^{\eta} \mathbf{u}, \text{ weakly in } L^2(0,T; L^2(\Omega)), ~~\text{as } k \to \infty.
\end{eqnarray*}

Now, by using the weak and strong convergences obtained in the section 4, we are going to prove the convergences for all the terms $I_i$
\begin{enumerate}
\item Terms $I_1$ and $I_2$: Follows from the convergences of functions $T_{\Delta t}J_{\Delta t,k}, \bb{u}_{\Delta t,k}, \bb{w}_{\Delta t,k}$ and $\bb{q}_{\Delta t,k}$ and from the weak convergence of $\nabla^{\tilde{\eta}_{\Delta t,k}} \mathbf{u}_{\Delta t,k}$ and convergence of $\nabla^{\widetilde{\eta}_{\Delta t,k}}\bb{q}$ which follows from the convergences of $\nabla \bb{q}_{\Delta t,k}$ and $A_{\Delta t,k}$. Since
\begin{eqnarray*}
||T_{\Delta t}\bb{u}_{\Delta t,k} - \bb{u}_{\Delta t,k}||_{L^2(0,T; L^2(\Omega))}^2 \leq C \sum_{n=1}^{N}||\bb{u}_{\Delta t,k}^{n+1} - \bb{u}_{\Delta t,k}^n||_{L^2(\Omega)}^2\Delta t \leq C\Delta t
\end{eqnarray*}
due to Lemma $\ref{111}$ $(6)$, function $T_{\Delta t} \bb{u}_{\Delta t,k}$ also converges to $\bb{u}$. 
\item Terms $I_3$ and $I_8$: Follows by partial integration on every sub-interval while being careful about left and right limits of $\mathbf{q}_{\Delta t,k}$ and $\psi_{\Delta t,k}$ at points $n\Delta t$ and by the mean-value Theorem (see \cite[section 7.2]{BorSun} or \cite[section 7.2]{multilayer}):
\begin{eqnarray*}
\int_0^T \int_{\Omega} T_{\Delta t} J_{\Delta t, k} \partial_t\widetilde{\mathbf{u}}_{\Delta t, k} \cdot \mathbf{q}_{\Delta t, k}   \to 
\int_0^T \int_{\Omega} \partial_t J \bb{u} \bb{q}  -\int_0^T \int_{\Omega} J_0 \mathbf{u}_{0} \partial_t \mathbf{q}(0) - \int_0^T \int_\Omega J \bb{u} \partial_t \bb{q}
\end{eqnarray*}
and
\begin{eqnarray*}
\int_0^T\int_\Gamma \widetilde{v}_{\Delta t, k} \cdot \psi_{\Delta t, k} \to    \int_{\Gamma} v_{0} \psi_{0}- \int_0^T\int_\Gamma  \partial_t \eta_{0} \cdot \psi_{0} 
\end{eqnarray*}

Here we also used the strong convergence of initial data in the corresponding norms.
\item Term $I_4$: Since 
\begin{eqnarray*}
\ddfrac{J_{\Delta t, k}-T_{\Delta t} J_{\Delta t,k}}{\Delta t} = \widetilde{\partial}_t \eta_{\Delta t,k},
\end{eqnarray*}
convergence of this term follows from the convergence of  $\widetilde{\partial}_t \eta_{\Delta t,k}$.
\item Terms $I_5,I_6$ and $I_7$: Directly from the weak convergences of the corresponding terms and the convergence of the term $\mathcal{F}(\eta_{\Delta t,k})$ in $H^{-2}(\Gamma)$.
\end{enumerate}
Notice that we had to take more regular test functions to pass the limit in the term $I_1$ since the function $\bb{w}_{\Delta t,k}$ converges only in $L^2(0,T; L^2(\Omega))$. On the other hand, since the function $\bb{w} = (0,0,\partial_t \eta) \in L^2(0,T; L^3(\Omega))$ (due to the imbedding of Sobolev spaces), and since the other limiting functions are regular enough, the limiting problem is still satisfied for the test functions from $Q^\eta(0,T)$ by the density argument.

Thus, we have proven the existence of a weak solution in the sense of the Definition 2.1. The energy estimate in the Theorem 2.1 follows from the Lemma $\ref{111}(1)$ and the coercivity estimate of the potential $\Pi$ given in assumption $(A3)$, by using the lower semicontinuity of norms. The remaining part is to study the life span of the solution. We follow the approach given in \cite[pp. 397-398]{time} (see also \cite[Theorem 7.1]{BorSun}). For initial data $X_0$ we solve the problem on the time interval $[0,T_0]$. Let 
\begin{eqnarray*}
\displaystyle\lim_{t \to T_0}\displaystyle\min_{ X \in \Gamma } \eta(t, X)+1=C_0>0
\end{eqnarray*}
Since $\eta$ is uniformly bounded in $W^{1,\infty}(0,T; L^2(\Gamma))\cap L^\infty(0,T; H_0^2(\Gamma))$ which is embedded into $C^{0, 1/4}([0,T]; H^{3/2}(\Gamma))$, we can choose $T_1>T_0$ such that
\begin{eqnarray*}
||\eta_{\Delta t,k}(T_1)-\eta_{\Delta t,k}(T_0)||_{
C(\overline{\Gamma})}\leq C||\eta_{\Delta t,k}(T_1)-\eta_{\Delta t,k}(T_0)||_{H^{3/2}(\Gamma)}
\leq 2C(T_1-T_0)^{1/4},
\end{eqnarray*}
so for $T_1-T_0 \leq \big(\frac{C_0}{2C}\big)^4$, we can prolong the solution defined in $[0,T_0]$ to be in $[0,T_1]$ and
\begin{eqnarray*}
\displaystyle\lim_{t \to T_1}\displaystyle\min_{ X \in \Gamma } \eta(t, X)+1=C_1> C_0/2.
\end{eqnarray*} 
Now we repeat this procedure and obtain a sequence of $C_n = \displaystyle\lim_{t \to T_n}\displaystyle\min_{ X \in \Gamma } \eta(t, X)+1$, and if $\displaystyle\lim_{n \to \infty} C_n = 0$, then we take $T = \displaystyle\lim_{n \to \infty} T_n$. Otherwise, there exists a $C^*>0$ such that $C_n \geq C^*$, for all $n \in \mathbb{N}$. Now, since
\begin{eqnarray*}
1+ \eta(T_{n+1},X) \geq 1+ \eta(T_n, X) - 2C(T_{n+1}-T_n)^{1/4} \geq C^* -  2C(T_{n+1}-T_n)^{1/4},
\end{eqnarray*}
we have that
\begin{eqnarray*}
T_{n+1}-T_n \geq  \Big(\frac{C^*}{1+ \eta(T_{n+1},X)}\Big)^4
\end{eqnarray*}
Since $\eta$ is uniformly bounded in $C([0,T] \times \overline{\Gamma} )$, we obtain that $T_{n+1}-T_n\geq C >0$, so $ \displaystyle\lim_{n \to \infty} T_n = \infty$. This finishes the proof of Theorem 1.1.

\subsection*{Appendix A: The plate models}
Here, we will give some concrete examples for the nonlinear elastic force $\mathcal{F}(\eta)$ that appears in the plate equation $\eqref{plateeq}$. The following models were studied in $\cite{igor,igor2,platesproofs}$.\\ \\
\textbf{The Kirchhoff model}: Here the nonlinear elastic force takes the form of the Nemytskii operator:
\begin{eqnarray*}
\mathcal{F}(\eta) = - \nu  \cdot \text{div} \big[ |\nabla \eta|^q  \nabla \eta - \mu |\nabla \eta|^r \nabla \eta \big] + f(\eta) -h(x),
\end{eqnarray*}
where $\nu \geq 0$, $q>r \geq 0$, and $\mu \in \mathbb{R}$ are parameters, $h\in L^2(\Gamma)$ and
\begin{eqnarray}\label{propertyforf}
f \in \text{Lip}_{loc}(\mathbb{R}) \quad \text{satisfies} \quad \lim\limits_{|s|\to \infty}\inf f(s)s^{-1} > -\lambda_1^2, 
\end{eqnarray}
where $\lambda_1$ is the first eigenvalue for the Laplacian with the Dirichlet boundary condition. The corresponding potential is
\begin{eqnarray}\label{potentialkirch}
\Pi(\eta) = \int_\Gamma \Phi(\eta) + \frac{\nu}{q+2} \int_{\Gamma} |\nabla \eta|^{q+2} - \frac{\nu \mu}{r+2} \int_\Gamma |\nabla \eta|^{r+2} - \int_\Gamma \eta h
\end{eqnarray}
where $\Phi(s) = \int_0^s f(x) dx$. To verify the assumptions $(A1)$ and $(A2)$, let's estimate the  difference (for $s \geq 0$)
\begin{eqnarray*}
|| \text{div} \big[ |\nabla \eta_1|^s  \nabla \eta_1 -  |\nabla \eta_2|^s  \nabla \eta_2\big] ||_{H^{-a}(\Gamma)} &\leq& C||  |\nabla \eta_1|^s  \nabla \eta_1 -  |\nabla \eta_2|^s  \nabla \eta_2 ||_{H^{1-a}(\Gamma)}  \\
 &\leq& C || \eta_1 - \eta_2||_{H^{2-\varepsilon}(\Gamma)} (||\eta_1||_{H^2(\Gamma)}^s + ||\eta_2||_{H^2(\Gamma)}^s)
\end{eqnarray*}
for $0< a \leq 1$ and $0< \varepsilon<a$ (see $\cite[\text{Section } 5.1]{platesproofs}$ for more details on the derivation of this inequality). This gives us that assumptions $(A1)-(A2)$ hold for any $a>0$ and $0<\varepsilon<a$. 

To prove the assumption $(A3)$, notice that the property $\eqref{propertyforf}$ gives us that there exists $0<\gamma<\lambda_1^2$ such that $\Phi(s) \geq -\gamma s^2/2-C$, for all $s$, and since there is a $\gamma'$ such that $\gamma< \gamma' <\lambda_1^2$ for which
\begin{eqnarray*}
- \int_\Gamma \eta h \geq - \frac{1}{2}\int_\Gamma (\gamma'-\gamma)\eta^2 - \frac{1}{ (\gamma'-\gamma)} h^2 \geq - \frac{1}{2}\int_\Gamma (\gamma'-\gamma)\eta^2  - C
\end{eqnarray*}
we obtain
\begin{eqnarray*}
\int_\Gamma \Phi(\eta)  - \int_\Gamma \eta h \geq -\frac{\gamma'}{2} \int_\Gamma \eta^2-C
\end{eqnarray*}
Noticing that the function $C_1 |x|^{q+2} - C_2|x|^{r+2}$, $C_1> 0,C_2 \in \mathbb{R}$ is bounded from below by a constant, one concludes that the two middle terms on the right hand side of $\eqref{potentialkirch}$ can be bounded from below by a constant. Next, by squaring the both sides of the harmonic eigenvalue problem, since $\lambda_1$ is its lowest eigenvalue, we can conclude that
\begin{eqnarray*}
||\Delta \eta ||_{L^2(\Gamma)}^2 \geq \lambda_1^2 || \eta ||_{L^2(\Gamma)}^2, \quad \forall \eta \in \text{H}_0^2(\Gamma)
\end{eqnarray*}
so the coercivity of the potential holds by combining the previous arguments for $\kappa = \gamma' /(2\lambda_1^2)<1/2$. The boundedness on bounded sets in $H_0^2(\Gamma)$ of potential can be proved easily by the means of the Sobolev imbedding Theorem, which gives us that the Kirchhoff model satisfies the assumption $(A3)$.\\ \\
\textbf{The von K\'{a}rm\'{a}n model}: Here, the nonlinear elastic force takes form
\begin{eqnarray*}
\mathcal{F}(\eta) = -[\eta, v(\eta)+F_0] - h
\end{eqnarray*}
where $F_0 \in H^4(\Gamma)$ and $h \in L^2(\Gamma)$ are given functions, and the von Karman bracket is defined as
\begin{eqnarray*}
[w,u] = \partial_{x_1}^2 w \cdot \partial_{x_2}^2 u+ \partial_{x_2}^2 w \cdot \partial_{x_1}^2 u - 2 \cdot \partial_{x_1 x_2}^2 w \cdot \partial_{x_1 x_2}^2 u, 
\end{eqnarray*}
where $v = v(\eta)$ satisfies the equation
\begin{eqnarray*}
\Delta v+[\eta,\eta] = 0 \text{ in } \Gamma, \quad \partial_\nu v = v =0 \text{ on } \partial\Gamma.
\end{eqnarray*}
The potential is given as
\begin{eqnarray*}
\Pi(\eta) = \frac{1}{4} \int_{\Gamma} \big[ |v(\eta)|^2 - 2([\eta,F_0] -2h)\eta \big].
\end{eqnarray*}
One can estimate the difference
\begin{eqnarray*}
|| [\eta_1 v(\eta_1)]-[\eta_2,v(\eta_2) ||_{H^{-\theta}(\Gamma)} \leq  C || \eta_1 - \eta_2||_{H^{2-\theta}(\Gamma)} (||\eta_1||_{H^2(\Gamma)}^s + ||\eta_2||_{H^2(\Gamma)}^s)
\end{eqnarray*}
for any $\theta \in [0,1]$ (see \cite[\text{Corollary 1.4.5}]{karmanplates}), so assumptions $(A1), (A2)$ hold for $a=0$ and $\varepsilon = 1$. The proof that the potential $\Pi(\eta)$ satisfies the assumption $(A3)$ can be found in $\cite[\text{Chapter 4}]{karmanplates}$. \\ \\
\textbf{The Berger model}: In this case, the nonlinear elastic force takes form
\begin{eqnarray*}
\mathcal{F}(\eta) =-\Big[ \nu \int_{\Gamma} |\nabla \eta|^2 dx - G  \Big] \Delta \eta - h,
\end{eqnarray*}
where $\nu >0$ and $G \in \mathbb{R}$ are parameters, $h\in L^2(\Gamma)$. The corresponding potential is 
\begin{eqnarray*}
\Pi(\eta) = \frac{\nu}{4} \Big[\int_\Gamma |\nabla \eta|^2 \Big]^2 - \frac{G}{2}\int_\Gamma |\nabla \eta|^2  - \int_\Gamma \eta h
\end{eqnarray*}
It is straightforward to prove that the assumptions $(A1),(A2)$ hold for say $a=0$ and $\varepsilon =1$. The coercivity of the potential can be proved by using the identity $x^2 \leq \delta x^4 + \frac{1}{4\delta}$, for any $\delta >0$, while the boundedness on the bounded sets in $H_0^2(\Gamma)$ follows easily, so the assumption $(A3)$ holds. For more details about this model, see $\cite[\text{Chapter 7}]{berger}$.\\ \\
\textbf{Other nonlinear models:} In \cite{BorSunNonLinear}, a fluid-structure interaction problem was studied between an incompressible, viscous fluid and a semilinear cylindrical Koiter membrane shell with inertia and a proof of  existence of a weak solution was given. We can easily prove that the  nonlinear elastic force in the structure equation in this model, due to its nice polynomial form (see \cite[equation (2.13)]{BorSunNonLinear}), satisfies the assumptions $(A1),(A2)$ for say $a=0$ and $\varepsilon = 1$, while its potential given in the equality \cite[equality (2.5)]{BorSunNonLinear} is equivalent to the $W^{1,4}$ norm (see \cite[section 3.1]{BorSunNonLinear}), so the assumption $(A3)$ follows immediately.\\

All the models mentioned above satisfy the assumption $(A2)$ for either $a=0$ or any $a>0$, but it is not clear if there are other relevant physical models that satisfy this assumption only for say some $a=2-\delta$ for a small $\delta >0$ (and at the same time do not satisfy it for some small $a \geq 0$). The freedom for choosing the parameter $a$ was left in the assumption since it makes the result more general but does not affect the proof of the main result in any way.

\subsection*{Appendix B}

\begin{lem}\label{extth}{\rm{(Extension by zero to $\Omega^M$)}}
Let $\eta \in C^{0,\alpha}(\Gamma)$ with $\alpha \in (0,1]$ be given such that $-1<C_1<\eta(t,X)<M(X)$, for all $X \in \Gamma$, where $M(X)$ is defined in $\eqref{functionM}$, and let $0<s\leq1$ and $s'=\frac{\alpha(3+2s)-3}{2}$. Let $A_{\eta}$ be extended by the mapping $A_{\eta}^{b} : \Omega^{b} \to \Omega^{b}$, where $\Omega^{b} = \Gamma \times (-1,b)$ and $b:=\max\limits_{X \in \Gamma}M(X)+1$, in the following way:
\begin{eqnarray*}
A_{\eta}^{b}(t,X,z) = \begin{cases}
A_\eta(t,X,z), \text{ if } (X,z) \in \Omega, \\ 
(0, \ddfrac{z}{b}(b-\eta(t,X))+(X,\eta(t,X)), \text{ elsewhere} 
\end{cases}
\end{eqnarray*}
We have the following results:
\begin{enumerate}
\item[(1)] The mapping $f \mapsto f \circ (A_{\eta}^b)^{-1}=:\hat{f}$ is continous from $W^{s,p}(\Omega^b)$ to $W^{s',p}(\Omega^b)$, where the continuity constant only depends on $||\eta(X)||_{C^{0,\alpha}(\Omega^m)}$, the size of $\Omega^b$ and lower bound of the Jacobian of the transformation $A_\eta^b$ (which can be chosen as $\min \{ C_1+1,1\}$).
\item[(2)] The extension by zero of $\tilde{f}_{|\Omega} \circ A_{\eta}^{-1}$ to $\Omega^M$ (function $(\tilde{f} \circ (A_{\eta}^b)^{-1})_{|\Omega^M}$) satisfies
\begin{eqnarray*}
||(\tilde{f} \circ (A_{\eta}^b)^{-1})_{|\Omega^M}||_{H^{s'}(\Omega^M)} \leq C||\tilde{f}||_{H^{s}(\Omega^M)} \leq C ||f||_{H^{s}(\Omega)}.
\end{eqnarray*}
\end{enumerate}
\end{lem}
\begin{proof}
Recall the following intrinsic seminorm (see \cite{adams})
\begin{eqnarray*}
|f|_{W^{s,p}(\Omega^b)}^p = \int_{\Omega^b} \int_{\Omega^b}\frac{|f(x) - f(y)|^p}{|x-y|^{3+sp}} dx dy <\infty.
\end{eqnarray*}
Now, we prove $(1)$ in the following way
\begin{eqnarray*}
|\hat{f}|_{W^{s',p}(\Omega^b)}^p &=&  \int_{\Omega^b} \int_{\Omega^b}\frac{|\hat{f}(x) - \hat{f}(y)|^p}{|x-y|^{3+ps'}} dx dy =  \int_{\Omega^b} \int_{\Omega^b}\frac{|f(X) - f(Y)|^p}{|(A_{\eta}^b)^{-1}(X)-(A_{\eta}^b)^{-1}(Y)|^{3+ps'}} \frac{1}{J^2} dX dY \\
 &= & \int_{\Omega^b} \int_{\Omega^b}\frac{|f(X) - f(Y)|^p}{|X-Y|^{3+ps}} \frac{|X-Y|^{3+ps}}{{|(A_{\eta}^b)^{-1}(X)-(A_{\eta}^b)^{-1}(Y)|^{3+ps'}}}       \frac{1}{J^2} dXdY \\
&\leq&  \int_{\Omega^b} \int_{\Omega^b}\frac{|f(X) - f(Y)|^p}{|X-Y|^{3+ps}} \frac{1}{J^2}dX dY  \sup\limits_{x,y \in \Omega^b}  \frac{|A_\eta^b(x)-A_\eta^b(y)|^{3+ps}}{{|x-y|^{3+ps'}}}  \\
&=&C(\Omega^b,C_1)|f|_{W^{s,p}(\Omega^b)}^p  ||A_\eta^b||_{C^{0,\alpha}(\Omega^b)} \leq C\Big(C_1, ||\eta(X)||_{C^{0,\alpha}(\Omega^b)}, \Omega^b\Big) |f|_{W^{s,p}(\Omega^b)}^p 
\end{eqnarray*}
To prove the statement $(2)$, we use the inequality \cite[\text{Theorem} 1.2.16]{traces}, which gives us that the extension of $f$ by $0$ to $\mathbb{R}^3$, denoted by $\tilde{f}$, satisfies
\begin{eqnarray*}
||\tilde{f}||_{H^{s}(\mathbb{R}^3)} \leq C(\partial\Omega) ||f||_{H^{s}(\Omega)},
\end{eqnarray*}
where $C(\partial\Omega)$ depends only on the Lipschitz constant of $\partial\Omega$, so the statement $(2)$ follows by previous inequality and the statement $(1)$.
\end{proof}
\begin{cor}
Let $\eta \in H_0^2(\Gamma)$ be such that $-1<C_1<\eta(t,X)<M(X)$, for all $X \in \Gamma$, where $M(X)$ is defined in $\eqref{functionM}$, and let $f\in H^1(\Omega^\eta)$. Then, extension by $0$ of function $f$ to $\Omega^M$ is continuous from $H^1(\Omega^\eta)$ to $H^{s'}(\Omega^M)$, for any $0<s' < 1/2$. The continuity constant only depends on $||\eta||_{H^2(\Gamma)}$ and $s'$, and the extension of a divergence-free function is still divergence free.
\end{cor}
\begin{proof}
By Sobolev embedding, $\eta \in C^{0,\alpha}(\Gamma)$, for any $\alpha<1$, so by Lemma $\ref{extth}$ $(1)$, $f \circ A_\eta \in H^m(\Omega)$, for any $m<1$ (for example, one can obtain the same continuity result from $H^{1}(\Omega^\eta)$ to $H^{m}(\Omega)$ as in Lemma $\ref{extth}$ $(1)$ by using the mapping $A_\eta$ and the domains $\Omega^\eta$ and $\Omega$). Now, by Lemma $\ref{extth}$ $(2)$, the claim follows.
\end{proof}


\vspace{.1in}
\noindent{\bf Acknowledgments:} This research was partially supported by National Natural Science Foundation of China (NNSFC) under Grant Nos. 11631008 and 11826019. The authors would like to express their graditude to the referees for carefully reading the manuscript and for the useful suggestions and insightful comments which helped significantly in improving the quality and the clarity of this paper.

  \end{document}